\documentclass[11pt]{article}
\usepackage{amsfonts}
\usepackage{amssymb}
\usepackage{mathrsfs}
\usepackage{amsmath}
\usepackage{amsthm}
\usepackage{amstext}
\usepackage{amsopn}
\usepackage{graphicx}
\usepackage{subfigure}
\usepackage{epstopdf}
\usepackage{enumerate}
\usepackage{ulem}
\usepackage{cite}
\newtheorem{theorem}{Theorem}[section]
\newtheorem{lemma}[theorem]{Lemma}
\newtheorem{proposition}[theorem]{Proposition}
\newtheorem{corollary}[theorem]{Corollary}

\theoremstyle{definition}
\newtheorem{definition}[theorem]{Definition}

\theoremstyle{remark}

\numberwithin{equation}{section}

\newcommand{\ba}{\begin{array}}
\newcommand{\ea}{\end{array}}

\newcommand{\Om}{\Omega}

\newcommand{\la}{\lambda}
\newcommand{\R}{{\mathbb{R}}}
\newcommand{\ds}{\displaystyle}

\begin{document}
\date{}
\title{ \bf\large{Existence and stability of steady state solutions of reaction-diffusion equations with nonlocal delay effect}\footnote{Partially supported by the NSFC of China (No.11671236), the Natural Science Foundation of Shandong Province of China (No. ZR2019MA006), the Fundamental Research Funds for the Central Universities (No.19CX02055A), China Scholarship Council and US-NSF grants DMS-1715651 and DMS-1853598.}}
\author{Wenjie Zuo\textsuperscript{1}\ \ Junping Shi\textsuperscript{2}\footnote{Corresponding author, Email: jxshix@wm.edu}
 \\
{\small \textsuperscript{1} Department of Mathematics, China University of Petroleum (East China),\hfill{\ }}\\
\ \ {\small Qingdao, Shandong,
266580, P.R.China\hfill{\ }}\\
{\small \textsuperscript{2} Department of Mathematics, William \& Mary,\hfill{\ }}\\
\ \ {\small Williamsburg, Virginia, 23187-8795, USA\hfill {\ }}}
\maketitle
\begin{abstract}
{A general reaction-diffusion equation with spatiotemporal delay and homogeneous Dirichlet boundary condition is considered. The existence and stability of positive steady state solutions  are proved via studying an equivalent reaction-diffusion system without nonlocal and delay structure and applying local and global bifurcation theory. The global structure of the set of steady states is characterized according to type of nonlinearities and diffusion coefficient. Our general results are applied to diffusive logistic growth models and Nicholson's blowflies type models. }

\end{abstract}

\section {Introduction}
Reaction-diffusion models have been used to describe the evolution of population density in biological or chemical problems, and the qualitative behavior of solutions to the models can be used to predict outcomes of natural or engineered biochemical events. Typical long term behavior of the models are the convergence to steady state solutions or time-periodic orbits, or formation of some particular spatiotemporal patterns. The reaction dynamics of the models often depends on the system states of past time, which induces time delays in the model equations. Realistic time delay terms in the model distribute over all past time, and due to the spatial structure and the diffusive nature of population, the time delay is also nonlocal over the space.

In this paper, we consider a general reaction-diffusion model with spatiotemporal nonlocal delay effect and Dirichlet boundary conditions:
\begin{equation}
  \label{1}
  \begin{cases}
    u_t(x,t)=d\Delta u(x,t)+F(\lambda, u(x,t),(g\ast\ast H(u))(x,t)),&x\in\Omega,~t>0,\\
    u(x,t)=0,&x\in\partial\Omega,~t>0,\\
    u(x,t)=\eta(x,t),&x\in\Omega,~t\in(-\infty,0],
  \end{cases}
\end{equation}
where $u(x,t)$ is the population density at time $t$ and location $x\in\Omega\subset \R^n$, $d>0$ is the diffusion coefficient, and the initial condition is assumed to be given for all past time; $F(\la, u, v)$ is a  nonlinear function depending on a parameter $\la$, the local population density $u(x,t)$, and a variable $v(x,t)$ representing past state of population density. Here the
past state of population density $v(x,t)$ is given by a form
\begin{equation}\label{3}
  v(x,t)=(g\ast\ast H(u))(x,t)=\int_{-\infty}^t\int_{\Omega}G(x,y,t-s)g(t-s)H(u(y,s))dyds,
\end{equation}
where the spatial weighing function $G(x,y,t-s)$ means the probability that an individual in location $y$ moves to location $x$ at a past time $t-s$, the temporal weighing function $g(t-s)$ characterizes the weight of past time $t-s$ in the entire past, and $H$ is a function of the state variable $u$. Here $G:\Omega\times \Omega\times (0,\infty)\to \R$ is a (generalized)  function or measure  and $g:[0,\infty)\to \R^+$ is a probability distribution function satisfying
\begin{equation}\label{1aa}
  \int_{\Omega} G(x,y,t) dy=1, \;\; x\in \Omega, \; t>0, \;\;\;\text{ and } \;\; \int_{0}^{\infty} g(t) dt=1.
\end{equation}
The nonlocal distributed delay term $g\ast\ast H(u)$ is a spatiotemporal average of the past state of density function $u$. Such nonlocal delay effect was first introduced in \cite{Britton1990} when $\Omega=\R^n$, and in \cite{Gourley2002}  when $\Omega$ is a bounded domain. See \cite{Gourley2004,Gourley2006,WangLiRuan2006} for more detailed explanation of the nonlocal delay in the population models.

In this paper, we assume that $G(x,y,t)$ is the Green's function of diffusion equation with Dirichlet boundary condition:
\begin{equation}\label{4}
  G(x,y,t)=\sum_{n=1}^{\infty}e^{-d\lambda_nt}\phi_n(x)\phi_n(y),
\end{equation}
where $\la_n$ is the $n$-th eigenvalue of  the following eigenvalue problem
\begin{equation*}\label{16}
  \begin{cases}
    -\Delta \phi(x)=\lambda \phi(x),&x\in\Omega,\\
    \phi(x)=0,&x\in\partial\Omega,
  \end{cases}
\end{equation*}
such that
\begin{equation*}
  0<\lambda_1\leq\lambda_2\leq\cdots\leq\lambda_n\leq\cdots\rightarrow+\infty,~\text{as}~n\rightarrow\infty,
\end{equation*}
and  $\phi_n(x)$ is the corresponding eigenfunction of $\lambda_n$ normalized so that \eqref{1aa} is satisfied. This assumption is consistent with the diffusive behavior of the population in the past time. On the other hand, the temporal distribution function is chosen to be
\begin{equation}\label{eq5}
  g_w(t)=\frac{1}{\tau}e^{-\frac{t}{\tau}},~ g_s(t)=\frac{t}{\tau^2}e^{-\frac{t}{\tau}},
\end{equation}
which are referred as weak kernel and strong kernel. When $G$ and $g$ take the forms in \eqref{4} and \eqref{eq5}, the model \eqref{1} is equivalent to a system of reaction-diffusion equations without nonlocal and delay effect (the precise equivalence is described in Section 2). For example, when the weak kernel is used, the new equivalent system is
\begin{equation}\label{388}
\begin{cases}
    u_t(x,t)=d\Delta u(x,t)+F(\lambda,u(x,t),v(x,t)),&x\in\Omega,~t>0,\\
   \ds v_t(x,t)=d\Delta v(x,t)+\frac{1}{\tau}(H(u(x,t))-v(x,t)),&x\in\Omega,~t>0,\\
   u(x,t)=v(x,t)=0,&x\in\partial\Omega,~t>0.
 \end{cases}
\end{equation}
We use established techniques for classical reaction-diffusion systems such as local and global bifurcation theory, linear stability analysis, nonlinear elliptic equations, and \textit{a priori} estimates to study \eqref{388}, which in turn provides information on
steady state solutions and dynamical behavior of reaction-diffusion equation with nonlocal delay effect \eqref{1}. Our results assume general form of the nonlinear functions $F$ and $H$, hence they can be applied to a wide variety of population growth models in the literature. In particular, we demonstrate our result by applying them to logistic type models \cite{Britton1990}, and Nicholson's blowflies type models \cite{WangLiRuan2006}.

Our results can be compared to a vast body of previous work on \eqref{1} with other choices of $G$ and $g$ as well as other boundary conditions. The spatiotemporal kernel $G$ can take the form: (A) $\delta(x-y)$ (local); (B) $K(x,y)$ (spatial); or (C) the one in \eqref{4} (diffusion). Special examples of (B) include: (B1) Green's function of stationary diffusion operator $-d\Delta +\mu$; or (B2) constant function. The delay distribution function $g$ can take the form: (a) $\delta(t-\tau)$ (discrete delay); or (b) $g_n(t)=\ds \frac{t^n e^{-t/\tau}}{\tau^{n+1}\Gamma(n+1)}$ (Gamma function of order $n$). Note that $g_w$ and $g_s$ defined in \eqref{eq5} are the Gamma function of order $0$ and $1$. Finally the boundary conditions can be: ($\alpha$) Dirichlet $u=0$; ($\beta$) Neumann $\ds\frac{\partial u}{\partial n}=0$; or ($\gamma$) periodic on $\R^n$. Various combinations of $G$, $g$ and boundary conditions have been used for \eqref{1}, and Table \ref{tab1} gives a partial list of references which consider \eqref{1} with these different choices of kernel functions and boundary conditions.

\begin{table}
\centering
\begin{tabular}{|c|c|c||c|c|c||c|c|c|}
  \hline
  ($\alpha$) & (a) & (b) & ($\beta$) & (a) & (b) & ($\gamma$) & (a) & (b) \\
   \hline
  (A) & \cite{BusenbergHuang1996,GreenStech1981,So1998,Su2009,SuWeiShi2012,YanLi2010,Yi2013} & \cite{Huang1998,Parrott1992,Shi2017} &(A)  & \cite{Memory1989,Shi2019,Yang1998,Yi2008,Yi2010,Yoshida1982} & \cite{Deng2015,GourleyRuan2000,ZuoSong2015} & (A)  &  &  \\
  (B) & \cite{ChenShi2012,ChenYu2016JDE,Guo2015,Guo2016,Yi2013}  &  & (B) & \cite{Ni2018} &  & (B) & \cite{Britton1989} &  \\
  (C) &  & \cite{ChenYuJDDE2016,Gourley2002} & (C) &  & \cite{Gourley2002,Su2014,ZuoSong2015b} & (C) &  & \cite{Britton1990}  \\
  \hline
\end{tabular}
\caption{References on dynamics of \eqref{1} with different combinations of $G$, $g$ and boundary conditions.}\label{tab1}
\end{table}

When the spatiotemporal kernel $G$ is a delta function $\delta(x-y)$ as type (A), the system \eqref{1} is spatially local. For discrete type delay (a), it has been shown that for Neumann boundary value problem, the positive steady state solution loses its stability via a Hopf bifurcation when the delay $\tau$ is large \cite{Memory1989,Shi2019,Yoshida1982}, while the same phenomenon is also proved for small amplitude positive steady state for Dirichlet boundary value problem \cite{BusenbergHuang1996,Su2009,SuWeiShi2012,YanLi2010}. A temporally oscillatory solution emerges from the Hopf bifurcation, and this solution is spatially non-homogeneous under Dirichelt boundary condition \cite{BusenbergHuang1996,Su2009,SuWeiShi2012,YanLi2010} or with spatial heterogeneity \cite{Shi2019}. Similar Hopf bifurcation and temporally oscillatory solution are also found when the delay is distributed one as type (b) \cite{GourleyRuan2000,Shi2017,ZuoSong2015}. When the kernel function $G$ is a spatial one as type (B), the system \eqref{1} is a nonlocal one. For discrete delay (a) and Dirichlet boundary condition, Hopf bifurcation and spatially non-homogeneous oscillatory solution bifurcating from small amplitude positive steady state have also been founded \cite{ChenShi2012,ChenYu2016JDE,Guo2015,Guo2016}. The rigorous proof of Hopf bifurcation and spatially non-homogeneous oscillatory solution bifurcating from large amplitude positive steady state remains an open question, although numerically it has been found in many cases.

For the diffusion kernel defined in \eqref{4} (C) and Gamma distribution function (b), it is found under Dirichlet boundary condition that the small amplitude positive steady state does not undergo Hopf bifurcation and it remains stable for $\tau>0$ \cite{ChenYuJDDE2016}. Same result holds for Neumann boundary condition and weak kernel, but Hopf bifurcation occurs for  Neumann boundary condition and strong kernel \cite{ZuoSong2015b}. This paper also considers the Dirichlet diffusion kernel defined in \eqref{4} (C) and weak kernel, and we show that for fixed $\tau>0$, the bifurcating positive steady state solution is usually locally asymptotically stable for $d\in (d^*(\tau)-\epsilon(\tau),d^*(\tau))$, where $d^*(\tau)$ is the bifurcation point and $\epsilon(\tau)$ is a small constant depending on $\tau$. So our results here again confirm the nonoccurrence of Hopf bifurcation for the diffusion kernel case and weak distribution kernel as indicated in \cite{ChenYuJDDE2016,ZuoSong2015b}. The results in this paper take an entirely different approach based on the equivalent system \eqref{388} and theory of semilinear elliptic systems, and it also holds for much general setting compared to the ones in \cite{ChenYuJDDE2016,ZuoSong2015b}. Some of our existence, stability and uniqueness results are of global nature (see Section 5 and 6).

Equation \eqref{1} has also been used to model biological invasion or spreading behavior, and traveling wave solutions of \eqref{1} with varies choices of $G$ and $g$ have been considered in,  for example, \cite{Ai2007,Ashwin2002,Gourley2000,Ma2001,So2001,WangLiRuan2006,WuZou2001}.

The rest of this paper is organized as follows. In Section 2, we prove the equivalence of the system (\ref{1}) with spatiotemporal delay and a system without nonlocal  and delay effect. Sections 3 is devoted to obtain the existence of the local bifurcated spatially nonhomogeneous steady-state solutions, and the stability of bifurcating solutions is shown in Section 4. In Section 5, the global bifurcation structure of positive steady state solutions is shown in two different scenarios, and a uniqueness of positive steady state result for one-dimensional case is shown in Section 6. In Section 7, we apply our main results to the Logistic type models and Nicholson's blowflies type equations.

\section{Equivalence of systems}
In this section we establish the equivalence of the reaction-diffusion system \eqref{1} with spatiotemporal delay given in  \eqref{4} and \eqref{eq5} and reaction-diffusion systems without delays. We will consider the cases of bounded domains and entire space $\R^n$.

\subsection{The bounded domain}
First we recall the following standard result for the linear parabolic equations.
\begin{lemma}
  \label{lem1}
  Let $\Om$ be a bounded domain in $\R^n$ with smooth boundary. Suppose that $f:\overline{\Omega}\times (t_0,+\infty)$ is continuous and $u\in C^{2,1}(\Omega\times[t_0,+\infty))\cap C^0(\bar{\Omega}\times[t_0,+\infty))$ satisfies
   \begin{equation}\label{5}
    \begin{cases}
      u_t(x,t)=d\Delta u(x,t)-ku(x,t)+f(x,t),&x\in\Omega,~t>t_0,\\
      Bu(x,t)=0,&x\in\partial\Omega,~t\geq t_0,\\
      u(x,t_0)=u_0(x),&x\in\Omega,
       \end{cases}
  \end{equation}
  where $Bu=u$, or  $Bu=\ds\frac{\partial u}{\partial n}+a(x)u$ with $a(x)\ge 0$.
Then
  \begin{equation}\label{15}
  u(x,t)=\int_{\Omega}G(x,y,t-t_0)e^{-k(t-t_0)}u_0(y)dy+\int_{t_0}^t\int_{\Omega}G(x,y,t-s)e^{-k(t-s)}f(y,s)dyds,
   \end{equation}
   where for any fixed $y\in\Omega,$ $G(x,y,t)$ is the Green function of the diffusion equation satisfying
\begin{equation*}
\begin{cases}
  G_t(x,y,t)=d\Delta_xG(x,y,t),&x\in \Omega,~t>0\\
  BG(x,y,t)=0,&x\in\partial\Omega,~t>0,\\
  G(x,y,0)=\delta(x-y).
  \end{cases}
\end{equation*}
\end{lemma}

\begin{proof}
  Denote by $\{(\mu_n,\varphi_n(x))\}_{n=1}^{\infty}$ the eigenvalues and the corresponding normalized eigenfunctions of
   \begin{equation*}\label{16}
  \begin{cases}
    -\Delta \varphi(x)=\mu \varphi(x),&x\in\Omega,\\
    B\varphi(x)=0,&x\in\partial\Omega.
  \end{cases}
\end{equation*}
The for the homogeneous equation
    \begin{equation*}\label{17}
   \begin{cases}
         v_t(x,t)=d\Delta v(x,t)-kv(x,t),&x\in\Omega,~t>t_0,\\
     Bv(x,t)=0,&x\in\partial\Omega,~t\geq t_0,\\
     v(x,t_0)=v_0(x),&x\in\Omega,
        \end{cases}
   \end{equation*}
the solution is   given by
  \begin{equation*}
    v(x,t)=\sum\limits_{n=1}^{\infty}c_ne^{-(d\mu_n+k)(t-t_0)}\varphi_n(x), \;\; c_n=\int_{\Omega}\phi_n(y)v_0(y)dy.
  \end{equation*}
This implies that
  \begin{equation*}
    \begin{split}
      v(x,t)&=\sum\limits_{n=1}^{\infty}\left(\int_{\Omega}\varphi_n(y)v_0(y)dy\right)e^{-(d\mu_n+k)(t-t_0)}\varphi_n(x)\\
      &=\int_{\Omega}\left(\sum\limits_{n=1}^{\infty}e^{-d\mu_n(t-t_0)}\varphi_n(x)\varphi_n(y)\right)e^{-k(t-t_0)}v_0(t_0)dy\\
      &=\int_{\Omega}G(x,y,t-t_0)e^{-k(t-t_0)}v_0(y)dy.
    \end{split}
  \end{equation*}
  By the Duhamel principle, it follows that the solution of the initial boundary value problem \eqref{5}
    is given by \eqref{15}.
\end{proof}
Now we have the following result regarding an entire solution $u(x,t)$ defined for $t\in (-\infty,+\infty)$:
\begin{lemma}
  \label{lem3}
Let $\Om$ be a bounded domain in $\R^n$ with smooth boundary. Suppose that $f:\overline{\Omega}\times (-\infty,+\infty)$ is continuous and $u\in C^{2,1}(\Omega\times(-\infty,+\infty))\cap C^0(\overline{\Omega}\times(-\infty,+\infty))$ satisfies
   \begin{equation*}\label{36}
    \begin{cases}
      u_t(x,t)=d\Delta u(x,t)-ku(x,t)+f(x,t),&x\in\Omega,~t\in(-\infty,+\infty),\\
      Bu(x,t)=0,&x\in\partial\Omega,t\in(-\infty,+\infty).
       \end{cases}
  \end{equation*}
  Then
  \begin{equation}\label{15a}
  u(x,t)=\int_{-\infty}^t\int_{\Omega}G(x,y,t-s)e^{-k(t-s)}f(y,s)dyds.
   \end{equation}
\end{lemma}
\begin{proof}
  For any fixed $t_0<t$, by Lemma \ref{lem1}, we have
  \begin{equation*}
    u(x,t)=h(x,t;t_0)+\int_{t_0}^t\int_{\Omega}G(x,y,t-s)e^{-k(t-s)}f(y,s)dyds,
  \end{equation*}
  where $h(x,t;t_0)\triangleq\ds\int_{\Omega}G(x,y,t-t_0)e^{-k(t-t_0)}u(y,t_0)dy$. And
  \begin{equation*}
   \|h(x,t;t_0)\|\leq\|u(\cdot,t_0)\|\int_{\Omega} G(x,y,t-t_0)dy e^{-k(t-t_0)}\leq\|u(\cdot,t_0)\|e^{-k(t-t_0)}.
   \end{equation*}
    Then $h(x,t;t_0)\rightarrow 0$ as $t_0\to-\infty$ and from the arbitrariness of $t_0$, we let $t_0\rightarrow-\infty$ and we obtain \eqref{15a}.
\end{proof}

%
By using Lemma \ref{lem3}, we have the following results on the equivalence of the two systems under the weak or strong distribution kernels.
\begin{proposition}
  \label{thm1}
  Suppose that the distributed delay kernel $g(t)$ is given by the weak kernel function $g_w(t)=\ds\frac{1}{\tau}e^{-\frac{t}{\tau}}$, and define
  \begin{equation*}\label{v}
    v(x,t)=(g_w\ast\ast H(u))(x,t)=\int_{-\infty}^t\int_{\Omega}G(x,y,t-s)g_w(t-s)H(u(y,s))dyds.
  \end{equation*}
\begin{enumerate}
   \item If $u(x,t)$ is the solution of \eqref{1}, then $(u(x,t),v(x,t))$ is the solution of
       \begin{equation}
  \label{6}
  \begin{cases}
    u_t(x,t)=d\Delta u(x,t)+F(\lambda,u(x,t),v(x,t)),&x\in\Omega,~t>0,\\
    \ds v_t(x,t)=d\Delta v(x,t)+\frac{1}{\tau}(H(u(x,t))-v(x,t)),&x\in\Omega,~t>0,\\
    Bu(x,t)=Bv(x,t)=0,&x\in \partial\Omega,~t>0,\\
    u(x,0)=\eta(x,0),&x\in\Omega,\\
    \ds v(x,0)=\frac{1}{\tau}\int_{-\infty}^0\int_{\Omega}G(x,y,-s)e^{\frac{s}{\tau}}H(\eta(y,s))dyds,&x\in\Omega.
      \end{cases}
\end{equation}
\item If $(u(x,t),v(x,t))$ is a solution of
 \begin{equation}
  \label{43}
  \begin{cases}
    u_t(x,t)=d\Delta u(x,t)+F(\lambda,u(x,t),v(x,t)),&x\in\Omega,~t\in\R,\\
    \ds v_t(x,t)=d\Delta v(x,t)+\frac{1}{\tau}(H(u(x,t))-v(x,t)),&x\in\Omega,~t\in\R,\\
    Bu(x,t)=Bv(x,t)=0,&x\in \partial\Omega,~t\in\R.\\
          \end{cases}
\end{equation}
   Then $u(x,t)$ satisfies (\ref{1}) such that $\eta(x,s)=u(x,s),~-\infty<s<0$.
In particular, if $(u(x),v(x))$ is a steady state solution of \eqref{6}, then $u(x)$ is a steady state solution of \eqref{1}; and if $(u(x,t),v(x,t))$ is a periodic solution of \eqref{43} with period $T$, then $u(x,t)$ is a periodic solution of \eqref{1} with period $T$.
 \end{enumerate}
\end{proposition}

\begin{proposition}
  \label{thm2} Suppose that the distributed delay kernel $g(t)$ is given by the strong kernel function $\ds g_s(t)=\frac{t}{\tau^2}e^{-\frac{t}{\tau}}$, and define
  \begin{equation}\label{v1}
    v(x,t)=(g_s\ast\ast H(u))(x,t)=\int_{-\infty}^t\int_{\Omega}G(x,y,t-s)g_s(t-s)H(u(y,s))dyds.
  \end{equation}
\begin{enumerate}
   \item If $u(x,t)$ is the solution of (\ref{1}), then $(u(x,t),v(x,t),w(x,t))$ is the solution of
       \begin{equation}
  \label{45}
  \begin{cases}
    u_t(x,t)=d\Delta u(x,t)+F(\lambda,u(x,t),v(x,t)),&x\in\Omega,~t>0,\\
   \ds v_t(x,t)=d\Delta v(x,t)+\frac{1}{\tau}(w(x,t)-v(x,t)),&x\in\Omega,~t>0,\\
   \ds  w_t(x,t)=d\Delta w(x,t)+\frac{1}{\tau}(H(u(x,t))-w(x,t)),&x\in\Omega,~t>0,\\
    Bu(x,t)=Bv(x,t)=Bw(x,t)=0,&x\in \partial\Omega,~t>0,\\
    u(x,0)=\eta(x,0),&x\in\Omega,\\
    \ds v(x,0)=\int_{-\infty}^0\int_{\Omega}G(x,y,-s)\frac{-s}{\tau^2}e^{\frac{s}{\tau}}H(\eta(y,s))dyds,&x\in\Omega,\\
    \ds w(x,0)=\int_{-\infty}^0\int_{\Omega}G(x,y,-s)\frac{1}{\tau}e^{\frac{s}{\tau}}H(\eta(y,s))dyds,&x\in\Omega.
      \end{cases}
\end{equation}
\item If $(u(x,t),v(x,t),w(x,t))$ is a solution of
 \begin{equation}
  \label{44}
  \begin{cases}
    u_t(x,t)=d\Delta u(x,t)+F(\lambda,u(x,t),v(x,t)),&x\in\Omega,~t\in\R,\\
   \ds v_t(x,t)=d\Delta v(x,t)+\frac{1}{\tau}(w(x,t)-v(x,t)),&x\in\Omega,~t\in\R,\\
   \ds  w_t(x,t)=d\Delta w(x,t)+\frac{1}{\tau}(H(u(x,t))-w(x,t)),&x\in\Omega,~t\in\R,\\
       Bu(x,t)=Bv(x,t)=Bw(x,t)=0,&x\in \partial\Omega,~t\in\R.
          \end{cases}
\end{equation}
   Then $u(x,t)$ satisfies (\ref{1}) with the strong kernel $g_s(t)$ such that $\eta(x,s)=u(x,s),~-\infty<s<0$.
   In particular, if $(u(x),v(x),w(x))$ is a steady state solution of (\ref{45}), then $u(x)$ is a steady state solution of (\ref{1}); if $(u(x,t),v(x,t),w(x,t))$ is a periodic solution of (\ref{44}) with period $T$, then $u(x,t)$ is  a periodic solution of (\ref{1}) with period $T$.
 \end{enumerate}
\end{proposition}

The proof of Proposition \ref{thm1} is immediate from Lemma \ref{lem3}, and the proof of Proposition \ref{thm2} follows from differentiating \eqref{v1} with respect to $t$ and elementary calculation. The equivalence of \eqref{1} and \eqref{45} has been first observed in \cite{Gourley2002}.

 \subsection{The whole space $\R^N$}
 Consider a general scalar reaction-diffusion equation with spatiotemporal delay in the entire space:
\begin{equation}\label{37}
u_t(x,t)=d\Delta u(x,t)+F(\lambda,u(x,t),(g**H(u))(x,t)),~x\in \R^N,~t\in \R.
\end{equation}
Here,
\begin{equation*}
 (g**H(u))(x,t)=\int_{-\infty}^t\int_{\R^N}G(x,y,t-s)g(t-s)H(u(y,s))dyds,
 \end{equation*}
where for $y\in\R^N$, $G(x,y,t)$ is a fundamental solution of
\begin{equation*}
\begin{cases}
G_t(x,y,t)=d\Delta_x G(x,y,t),&x\in\R^N,\; t>0,\\
 G(x,y,0)=\delta(x-y),&x\in\R^N, \; t>0.
\end{cases}
\end{equation*}

By using the similar method as Propositions \ref{thm1} and \ref{thm2}, we can prove the following results on equivalence of \eqref{37} and associated systems:
\begin{proposition}\label{lem8}
\begin{enumerate}
\item If $(u(x,t),v(x,t))$ is a solution of
\begin{equation*}\label{38}
\begin{cases}
    u_t(x,t)=d\Delta u(x,t)+F(\lambda,u,v),&x\in\R^N,~t\in\R,\\
   \ds v_t(x,t)=d\Delta v(x,t)+\frac{1}{\tau}(H(u(x,t))-v(x,t)),&x\in\R^N,~t\in \R,
 \end{cases}
\end{equation*}
then $u(x,t)$ is also a solution of (\ref{37}) with the weak kernel $\ds g_w(t)=\frac{1}{\tau}e^{-\frac{t}{\tau}}$.

\item If $(u(x,t),v(x,t),w(x,t))$ is a  solution of
\begin{equation*}
  \label{38}
  \begin{cases}
    u_t(x,t)=d\Delta u(x,t)+F(\lambda,u,v),&x\in\R^N,~t\in\R,\\
   \ds v_t(x,t)=d\Delta v(x,t)+\frac{1}{\tau}(w(x,t)-v(x,t)),&x\in\R^N,~t\in \R,\\
   \ds w_t(x,t)=d\Delta w(x,t)+\frac{1}{\tau}(H(u(x,t))-w(x,t)),&x\in\R^N,~t\in \R,\\
          \end{cases}
\end{equation*}
       then $u(x,t)$ is also a  solution of (\ref{37}) with the strong kernel $\ds g_s(t)=\frac{t}{\tau^2}e^{-\frac{t}{\tau}}$.
 \end{enumerate}
\end{proposition}
Note that the equivalence of systems is valid for any solution defined for all $t\in \R$, which include steady state solutions, periodic solutions, and also traveling wave solutions. This equivalence was first observed in \cite{Britton1990}. In this paper, we only consider the bounded domain case.

\section{Existence and local bifurcation of steady state solutions }

In this section, we consider the existence of positive steady state solution of the system  (\ref{1}) with a weak kernel subject to Dirichlet boundary condition. The strong kernel case can be considered similarly but will not be considered here. By Theorem \ref{thm1}, we only need to consider the steady state solutions of the equivalent system (\ref{6}), which are the solutions of system of semilinear elliptic system:
\begin{equation}\label{8a}
  \begin{cases}
    d\Delta u(x)+F(\lambda,u(x),v(x))=0,&x\in\Omega,\\
   \ds d\Delta v(x)+\frac{1}{\tau}(H(u(x))-v(x))=0,&x\in\Omega,\\
      u(x)=v(x)=0,&x\in\partial\Omega.
  \end{cases}
\end{equation}
In the following we always assume that $d>0$, $\tau>0$ and $\lambda\ge 0$. We use bifurcation method with parameter $d$ to prove the existence of positive solutions to \eqref{8a}. Note that a bifurcation analysis can also be conducted using parameter $\lambda$ with a fixed $d$. So in the following we assume $F(\lambda,u,v)\equiv F(u,v)$ as $\lambda$ is fixed, so we consider
\begin{equation}\label{8}
  \begin{cases}
    d\Delta u(x)+F(u(x),v(x))=0,&x\in\Omega,\\
   \ds d\Delta v(x)+\frac{1}{\tau}(H(u(x))-v(x))=0,&x\in\Omega,\\
      u(x)=v(x)=0,&x\in\partial\Omega.
  \end{cases}
\end{equation}
 We assume that the nonlinearities $F(u,v)$ and $H(u)$ in \eqref{8} satisfy
\begin{enumerate}
   \item[\bf(A1)] There exists a $\delta>0$ such that $F:  U_{\delta}\times U_{\delta}\to \R$ and $H:U_{\delta}\to \R$ are $C^2$ functions, where  $U_{\delta}=\{y\in \R: |y|<\delta\}$;

\item[\bf(A2)] $F(0,0)=0$, $H(0)=0$ and $H'(0)>0$.

\end{enumerate}
In the following, the first and second derivatives of $F$ and $H$ are denoted by
    \begin{equation}\label{41}
\begin{split}
 &F_u(0,0)=a,~F_v(0,0)=b,~H'(0)=k>0,\\
 &F_{uu}(0,0)=p,~F_{uv}(0,0)=q,~F_{vv}(0,0)=r,~H''(0)=l.
 \end{split}
\end{equation}
From {\bf (A2)}, it is known that $(u,v)=(0,0)$ is a trivial solution of \eqref{8} for any $d,\tau>0$.
Let $X=W^{2,p}(\Om)\times W^{1,p}_0(\Om)$ for $p>n$, and let $Y=L^p(\Om)$.
For the bifurcation of positive solutions of \eqref{8}, fixing $\tau>0$, we define a nonlinear mapping $W: \R\times X^2\rightarrow Y^2$ by
\begin{equation}\label{W}
  W(d,u,v)=\left(\begin{array}{c}
    d\Delta u+F(u,v)\\
    d\Delta v+\ds\frac{1}{\tau}(H(u)-v)\\
      \end{array}\right).
\end{equation}
Then a solution $(d,u,v)$ of \eqref{8} is equivalent to $W(d,u,v)=(0,0)^T$.

Our main result on the local bifurcation of positive solutions of \eqref{8} is as follows:

\begin{theorem}\label{thm:3.1}
  Suppose that $\tau>0$ is fixed, the conditions {\bf{(A1)}} and {\bf{(A2)}} hold, and also
  \begin{enumerate}
    \item[\bf(A3)] $a+bk>0$.
\end{enumerate}
  Define
  \begin{equation}\label{40}
   d^*(\tau)=\frac{1}{2\lambda_1\tau}(a\tau-1+\sqrt{(a\tau+1)^2+4b\tau k}),
    \end{equation}
    where $\la_1$ is the principal eigenvalue of $-\Delta$ in $H^1_0(\Om)$ with corresponding eigenfunction $\phi_1(x)>0$.
   Then
   \begin{enumerate}
   \item $d=d^*=d^*(\tau)$ is the unique bifurcation point of the system (\ref{8}) where positive solutions of \eqref{8} bifurcate from the line of trivial solutions $\Gamma_0=\{(d,0,0):d>0\}$.
   \item Near $(d,u,v)=(d^*,0,0)$, there exists  $\delta_1>0$ such that all positive solutions of (\ref{8}) near the bifurcation point lie on a smooth curve $ \Gamma_1=\{(d(s),u_d(s,\cdot),v_d(s,\cdot)): s\in(0,\delta_1)\}$
             with $d(s)=d^*+d'(0)s+s^2 z_0(s)$, $(u_d(s,\cdot),v_d(s,\cdot))=s(1,M)\phi_1(\cdot)+s^2(z_1(s,\cdot),z_2(s,\cdot))$, where
           \begin{equation}\label{13}
            M=\frac{2k}{a\tau+1+\sqrt{(a\tau+1)^2+4b\tau k}},
           \end{equation}
      such that $z_0:(0,\delta_1)\to \R$ and $z_1,z_2:(0,\delta_1)\to X$ are smooth functions satisfying $z_i(0)=0$ for $i=0,1,2$. Moreover
      \begin{equation}\label{first}
      \ds d'(0)=\frac{[k(p+2qM+rM^2)+bMl]\ds\int_{\Omega}\phi_1^3(x)dx}{2\lambda_1(k+M^2b\tau)
             \ds\int_{\Omega}\phi^2_1(x)dx}.
      \end{equation}
   \end{enumerate}
\end{theorem}

\begin{proof}
Let $W$ be defined as in \eqref{W}.  Then from {\bf (A1)}, $W$ is twice differentiable in $\R\times X_{\delta}^2$, where $X_{\delta}$ is an open neighborhood of $0$ in $X$.  The Fr\'echet derivative of $W$ in variable $(u,v)$ is
\begin{equation}\label{51}
  W_{(u,v)}(d,u,v)\left(
                    \begin{array}{c}
                      \xi_1 \\
                      \xi_2 \\
                    \end{array}
                  \right)
  =\left(\begin{array}{c}
  d\Delta \xi_1+F_u(u,v)\xi_1+F_v(u,v)\xi_2\\
  d\Delta \xi_2+\ds\frac{1}{\tau}(H'(u)\xi_1-\xi_2)\\
    \end{array}\right),
\end{equation}
and in particular when $(u,v)=(0,0)$, then
\begin{equation}\label{51}
  W_{(u,v)}(d,0,0)\left(
                    \begin{array}{c}
                      \xi_1 \\
                      \xi_2 \\
                    \end{array}
                  \right)
  =\left(\begin{array}{c}
  d\Delta \xi_1\\
  d\Delta \xi_2\\
    \end{array}\right)+A\left(
                    \begin{array}{c}
                      \xi_1 \\
                      \xi_2 \\
                    \end{array}
                  \right),
\end{equation}
where $A$ is defined by
\begin{equation}\label{A}
  A=\left(
      \begin{array}{cc}
        a & b \\
        \ds\frac{k}{\tau} & \ds-\frac{1}{\tau} \\
      \end{array}
    \right).
\end{equation}
The eigenvalues of $A$ satisfy the characteristic equation
\begin{equation*}
  \mu^2-\left(a-\frac{1}{\tau}\right)\mu-\frac{a+bk}{\tau}=0.
\end{equation*}
From {\bf (A3)}, we have $a+bk>0$, then it is easy to see that $A$ has a unique positive eigenvalue $\mu_1>0$
defined by
\begin{equation}\label{mu1}
  \mu_1=\frac{1}{2\tau}(a\tau-1+\sqrt{(a\tau+1)^2+4b\tau k})
\end{equation}
with a positive eigenvector $(1,M)$ where $M$ is defined in \eqref{13}. From the implicit function theorem, if $d>0$ is a bifurcation point for positive solutions of \eqref{8} from the line of trivial solutions, then $W_{(u,v)}(d,0,0)$ is not invertible. That is, the null space $N(W_{(u,v)}(d,0,0))\ne \{0\}$. From Fourier theory, we must have $d=\mu_1/\la_n$, where $\la_n$ is an eigenvalue of $-\Delta$ in $H^1_0(\Om)$. Since $\phi_1$ is the only eigenfunction which does not change sign in $\Omega$, then the only possible bifurcation point for positive solutions is $d=d^*=\mu_1/\la_1$ which is given by \eqref{40}.

At $(d^*,0,0)$, it is easy to compute the kernels of the linearized operator $W_{(u,v)}(d^*,0,0)$ and associated adjoint operator $W^*_{(u,v)}(d^*,0,0)$ respectively:
\begin{equation*}
  N(W_{(u,v)}(d^*,0,0))=span\{(1,M)\phi_1\}, \;\;
  N(W^*_{(u,v)}(d^*,0,0))=span\{(1,Mb\tau/k)\phi_1\}.
\end{equation*}
And the range of the operator $W_{(u,v)}(d^*,0,0)$ is described by the following form:
\begin{equation*}
  R(W_{(u,v)}(d^*,0,0))=\left\{(g_1,g_2)\in Y^2:~\int_{\Omega}\left(k g_1(x)+Mb\tau g_2(x)\right)\phi_1(x)dx=0\right\}.
\end{equation*}
Moreover we have
\begin{equation*}
  W_{d(u,v)}(d^*,0,0)[(1,M)\phi_1]=-\la_1(1,M)\phi_1\not\in  R(W_{(u,v)}(d^*,0,0)),
\end{equation*}
as $\ds \int_{\Omega}\left(k +M^2 b\tau \right)\phi_1^2 dx>0$ since $k +M^2 b\tau=M\sqrt{(a\tau+1)^2+4b\tau k}>0$. Now applying \cite[Theorem 1.7]{crandall1971bifurcation}, we conclude that the set of positive solutions to (\ref{8}) near $(d^*,0,0)$ is a smooth curve $\Gamma_1=\{(d(s),u_d(s,\cdot),v_d(s,\cdot):s\in(0,\delta)\}$ satisfying $d(0)=d^*$ with $d(s)=d^*+d'(0)s+s^2 z_0(s)$, $(u_d(s,\cdot),v_d(s,\cdot))=s(1,M)\phi_1(\cdot)+s^2(z_1(s,\cdot),z_2(s,\cdot))$, $z_0:(0,\delta)\to \R$ and $z_1,z_2:(0,\delta)\to X$ are smooth functions satisfying $z_i(0)=0$ for $i=0,1,2$.  Furthermore, $d'(0)$ can be calculated by (see, for example \cite{Shi1999}),
\begin{equation*}
\begin{split}
\ds d'(0)=&-\frac{\langle \zeta,W_{(u,v)(u,v)}(\lambda,d^*,0,0)((1,M)^T\phi_1(x))^2\rangle}
{2\langle \zeta, W_{d(u,v)}(\lambda,d^*,0,0)(1,M)^T\phi_1(x)\rangle}\\
\ds =&-\frac{\langle \zeta,(p+2qM+rM^2,l/\tau)\phi_1^2(x)\rangle}{2\langle \zeta,-\lambda_1(1,M)\phi_1(x)\rangle}\\
\ds =&\frac{
[k(p+2qM+rM^2)+Mbl]
\ds\int_{\Omega}\phi_1^3(x)dx}{\ds 2\lambda_1(k+M^2bl)\int_{\Omega}\phi^2_1(x)dx},
  \end{split}
\end{equation*}
where $\zeta$ is a linear function on $Y^2$ defined as
\begin{equation*}\langle \zeta,[f_1,f_2]\rangle=\int_{\Omega}\left(f_1(x)+f_2(x)\frac{Mb\tau}{k}\right)\phi_1(x)dx.
\end{equation*}
Obviously, if $d'(0)>0~(\text{resp.}\; d'(0)<0)$, the $d(s)>d^*~(\text{resp.}\; d(s)<d^*)$ for $s\in (0,\delta_1)$, and  nonconstant positive solutions exist for $d\in(d^*,d^*+\epsilon)~(\text{resp.}\; d\in(d^*-\epsilon,d^*))$. 
\end{proof}

We notice that the bifurcation point $d=d^*(\tau)$ depends on the parameter $\tau$ (which  is related to the delay in the original spatiotemporal model). We can characterize the bifurcation point (or threshold diffusion rate) $d=d^*(\tau)$ in more details:

\begin{proposition}\label{lem10}
Suppose that the conditions {\bf{(A1)}}-{\bf{(A3)}} hold, and let $d^*(\tau)$ be the bifurcation point defined in Theorem \ref{thm:3.1}. Then
  \begin{enumerate}
   \item if $b=0$, then $d^*(\tau)=\ds\frac{a}{\lambda_1}$ which is independent of $\tau$;

    \item if $a\leq0$, then $d^*(\tau)$ is strictly decreasing in $\tau$;

     \item if $a>0$ and $b>0$, then $d^*(\tau)$ is strictly decreasing in $\tau$; if $a>0$ and $b<0$, then $d^*(\tau)$ is strictly increasing in $\tau$;
    \item
    \begin{equation*}
    \lim\limits_{\tau\rightarrow0^+}d^*(\tau)=\frac{a+bk}{\lambda_1},
    ~~\lim\limits_{\tau\rightarrow+\infty}d^*(\tau)=\frac{1}{2\lambda_1}(a+|a|)=
    \begin{cases}
      0,&\textrm{ if}\; a<0,\\
      \ds\frac{a}{\lambda_1},&\textrm{ if}\; a>0.
    \end{cases}
    \end{equation*}
  \end{enumerate}
\end{proposition}
\begin{proof}
1. It is easy to verify that $d^*(\tau)=\ds \frac{a}{\lambda_1}$ when $b=0$ from \eqref{40}.

2. Note that $d^*(\tau)=\ds\frac{2(a+bk)}{\lambda_1L(\tau)}$ where $L(\tau)=1-a\tau+\sqrt{(a\tau+1)^2+4b\tau k}$.
Then \begin{equation}\label{L}
    L'(\tau)=\frac{a(a\tau+1)+2bk-a\sqrt{(a\tau+1)^2+4b\tau k}}{\sqrt{(a\tau+1)^2+4b\tau k}}>0,
  \end{equation}
  as $a\leq0$, $b>0$, $k>0$ and $a+bk>0$ from assumptions and {\bf{(A2)}}, {\bf{(A3)}}. Thus $d^*(\tau)$ is decreasing in $\tau$.

3. If $a>0$ and $b>0$, from \eqref{L} we have
    \begin{equation}\label{L2}
    L'(\tau)=\frac{4bk(a+bk)}{\sqrt{(a\tau+1)^2+4b\tau k}[a(a\tau+1)+2bk+a\sqrt{(a\tau+1)^2+4b\tau k}]}>0,
  \end{equation}
so $d^*(\tau)$ is decreasing in $\tau$. On the other hand, if $a>0$, $b<0$ and $a+bk>0$, then $L'(\tau)<0$ from \eqref{L2} and $d^*(\tau)$ is increasing in $\tau$.

4. \begin{equation*}
\begin{split}
\lim\limits_{\tau\rightarrow0^+}d^*(\tau)&=\lim\limits_{\tau\rightarrow0^+}\frac{2(a+bk)}{\lambda_1L(\tau)}=\frac{a+bk}{\lambda_1},\\
\lim\limits_{\tau\rightarrow+\infty}d^*(\tau)&
=\lim\limits_{\tau\rightarrow+\infty}\frac{1}{2\lambda_1}\left(a-\frac{1}{\tau}+\sqrt{(a+\frac{1}{\tau})^2+\frac{4bk}{\tau}}\right)
=\frac{1}{2\lambda_1}(a+|a|).
   \end{split}
 \end{equation*}
 \end{proof}


\section{Stability of bifurcating steady states }
In Section 3, we have shown that for fixed $\tau$, non-constant steady state solutions $(d(s),u(s),v(s))\in \Gamma_1$ bifurcate from the line of trivial solutions near $d=d^*$ under the conditions {\bf (A1)}-{\bf (A3)}. In this section, we investigate the local stability of the bifurcating steady state solutions by applying the method in \cite{crandall1973}.

Consider an equation:
\begin{equation*}\label{48}
  W(d,u,v)=0,
\end{equation*}
where $W:\mathscr{S}\times V\rightarrow Y$ is a twice continuously Fr$\acute{e}$chet differentiable mapping and $X,Y$ are Banach spaces; $V$ is an open neighborhood of $(0,0)$ in $X$, $\mathscr{S}=(a,b)\subset \R$. We first recall some necessary definitions and results in \cite{crandall1973}.
\begin{definition}\cite[Definition 1.2]{crandall1973}
  Let $T,K\in B(X,Y)$, where $B(X,Y)$ denotes the set of bounded linear maps from $X$ to $Y$. Then $\mu\in\R$ is a $K-$simple eigenvalue of $T$ if
  \begin{equation*}
    dim N(T-\mu K)=codim R(T-\mu K)=1,
  \end{equation*}
  and if $N(T-\mu K)=span\{x_0\}$, $Kx_0\not\in R(T-\mu K)$.
\end{definition}
In our case, for $X=W^{2,p}(\Om)\cap W^{1,p}_0(\Om)$ and $Y=L^p(\Om)$, the mapping $K:X\to Y$ is simply the inclusion map $K(u)=u$.
Then the Theorem of Exchange of Stability in \cite[Theorem 1.16]{crandall1973} can be stated as follows adapting to \eqref{8}.
\begin{theorem}\label{49}
Assume the conditions in Theorem \ref{thm:3.1} are satisfied, and let $\Gamma_0,\Gamma_1$ be the line of trivial solutions and the curve of non-constant solutions of \eqref{8}.
Then the following results are true:
  \begin{enumerate}
    \item There exist open neighbourhoods $\tilde{\mathscr{I}},~\tilde{\mathscr{J}}$ of $d^*$ and $0$ and continuously differentiable functions $r:\tilde{\mathscr{I}}\rightarrow\R,~\mu:\tilde{\mathscr{J}}\rightarrow\R,~z:\tilde{\mathscr{I}}\rightarrow X,~w:\tilde{\mathscr{J}}\rightarrow X$ satisfying
        \begin{equation*}
        \begin{split}
          W_{(u,v)}(d,0,0)z(d)=r(d)K z(d),~~&d\in\tilde{\mathscr{I}},\\
          W_{(u,v)}(d(s),u(s,\cdot),v(s,\cdot))w(s)=\mu(s)Kw(s),~~&s\in\tilde{\mathscr{J}},
          \end{split}
        \end{equation*}
    where $r(d^*)=\mu(0)=0$, $K:X\to Y$ is defined by $K(u)=u$.
    \item $r'(d^*)\neq0$ and near $s=0$, $\mu(s)$ and $-sd'(s)r'(d^*)$ have the same zeros and the same sign whenever $\mu(s)\neq0$. More precisely,
        \begin{equation*}
          \lim\limits_{s\rightarrow0}\frac{-sd'(s)r'(d^*)}{\mu(s)}=1.
        \end{equation*}
  \end{enumerate}
\end{theorem}
Then we have the following stability result for \eqref{6} by applying Theorem \ref{49}:
\begin{theorem}\label{50}
Assume the conditions in Theorem \ref{thm:3.1} are satisfied. Then
  \begin{enumerate}
   \item when $k(p+2qM+rM^2)+Mbl<0$, the positive steady state solution $(u(s,\cdot),v(s,\cdot))$ obtained in Theorem \ref{thm:3.1} is locally asymptotically stable with respect to \eqref{6} for $s\in (0,\delta_2)$ and $d(s)\in(d^*-\epsilon,d^*)$;

    \item when $k(p+2qM+rM^2)+Mbl>0$, the positive steady state solution $(u(s,\cdot),v(s,\cdot))$ obtained in Theorem \ref{thm:3.1} is unstable with respect to \eqref{6} for $s\in (0,\delta_2)$ and $d(s)\in(d^*,d^*+\epsilon)$.
        \end{enumerate}
\end{theorem}
\begin{proof} From \eqref{first}, $k+M^2 b\tau>0$ and that $\phi_1>0$, we have ${\rm Sign}(d'(0))={\rm Sign}[k(p+2qM+rM^2)+Mbl]$. On the other hand, it is easy to see that from (\ref{51}), for $z(d)=s(1,M)^T\phi_1(x)$, we have \begin{equation*}
    W_{(u,v)}(d,0,0)z(d)=s\left(\begin{array}{c}d\Delta\phi_1(x)+a\phi_1(x)+bM\phi_1(x)\\Md\Delta\phi_1(x)+\ds\frac{1}{\tau}(k\phi_1(x)-M\phi_1(x))\end{array}\right)=
    r(d)s\left(\begin{array}{c}1\\M\end{array}\right)\phi_1(x).
\end{equation*}
That is, $r(d)=-d\lambda_1+a+bM$ hence $r'(d^*)=-\lambda_1<0$. Therefore from Theorem \ref{49} part 2, we have ${\rm Sign} (\mu(s))={\rm Sign}(d'(s))={\rm Sign}(d'(0))={\rm Sign}[k(p+2qM+rM^2)+Mbl]$ for $s\in (0,\delta_2)$. In particular, when $k(p+2qM+rM^2)+Mbl<0$, $\mu(s)<0$  and $(u(s,\cdot),v(s,\cdot))$  is locally asymptotically stable with respect to \eqref{6}; and when $k(p+2qM+rM^2)+Mbl>0$, $\mu(s)>0$ and  $(u(s,\cdot),v(s,\cdot))$  is unstable.
\end{proof}

The stability result in Theorem \ref{50} implies the non-occurrence of Hopf bifurcations when the parameter $(d,\tau)$ is in the range described in Theorem \ref{50}.
\begin{corollary}\label{cor:44}
Suppose that $\tau>0$ is fixed and the conditions {\bf (A1)}-{\bf (A3)} are satisfied, and let $d^*(\tau)$ be defined as in \eqref{40}. Then  there is no Hopf bifurcation occurring for the positive steady state $d\in(d^*(\tau)-\epsilon(\tau),d^*(\tau))$ when $k(p+2qM+rM^2)+Mbl<0$.
\end{corollary}
One should be cautious that the results in Corollary \ref{cor:44} is obtained for a fixed $\tau>0$ and $d^*(\tau)$, $\epsilon(\tau)$ both depend on the value of $\tau$. In other situations especially the discrete delay case, the steady state is independent of delay and Hopf bifurcation could occur when the delay value increases \cite{BusenbergHuang1996,ChenShi2012,Su2009}.

\section{Global bifurcation of steady states }
In Section 3, we only consider the existence of positive steady state solutions of (\ref{8}) near the bifurcation points using local bifurcation theory. Next we consider the global bifurcation of positive steady states of (\ref{8}) in two different cases. Here we assume $F$ and $H$ satisfy the following condition not restricted to neighborhoods of zeros:
\begin{enumerate}
   \item[\bf(A1')]  $F:  \R^+\times \R^+ \to \R$ and $H:\R^+\to \R$ are $C^2$ functions.
\end{enumerate}

\subsection{Case 1: $a=F_u(0,0)>0$.}
Here we further assume the following condition holds:
\begin{enumerate}
   \item[\bf(A4)] There exist a continuous function $F_1:\bar{\R}_+\rightarrow \R$ and positive constants $K_0>0$ and $u^*>0$ such that $F(u,v)\leq F_1(u)u$ for $(u,v)\in\bar{\R}_+\times\bar{\R}_+$, and $F_1$ satisfies $F_1(u^*)=0$ and $0<F_1(u)<K_0$ for $u\in(0,u^*)$ and $F_1(u)<0$ for $u>u^*$.
%
\end{enumerate}

First we have the following \textit{a priori} bound for the steady state solutions when {\bf (A4)} is satisfied.
\begin{lemma}\label{lem5}
Suppose the conditions {\bf(A1'), (A2), (A4)} hold and $(u,v)$ is a nonnegative solution of (\ref{8}). Then
\begin{equation}\label{H*}
    0\leq u(x)\leq u^*,\;\; 0\leq v(x)\leq \max_{0\leq u\leq u^*} H(u):=H^*,
\end{equation}
where $u^*$ is defined in {\bf (A4)}.
\end{lemma}

\begin{proof}
If $u(x)\equiv 0$, then $v(x)\equiv 0$ and the result is obviously true. Hence we assume that $u(x)>0$ for $x\in \Omega$ from the maximum principle. Let $x_0\in\Omega$ such that $u(x_0)=\max\limits_{x\in\bar{\Omega}}u(x)>0$. Then  from the maximum principle,  the first equation of (\ref{8}) and {\bf (A4)}, we have that
\begin{equation*}
0\leq-d\Delta u(x_0)=F(u(x_0),v(x_0))\leq F_1(u(x_0))u(x_0).
\end{equation*}
This implies that $F_1(u(x_0))\geq0$, and from {\bf (A4)}, we have $0<u(x_0)\leq u^*$ and consequently $0<u(x)<u^*$ in $\Omega$ from the strong maximum principle.

Since $u(x)>0$, then $v(x)=(-d\Delta+\tau^{-1})^{-1}(\tau^{-1}H(u))>0$ for $x\in \Omega$. Let $x_1\in\Omega$ such that $v(x_1)=\max\limits_{x\in\bar{\Omega}}v(x)>0$. Then  from the maximum principle and  the second equation of (\ref{8}),
we have that
\begin{equation*}
0\leq-d\Delta v(x_1)=\frac{1}{\tau}[H(u(x_1))-v(x_1)],
\end{equation*}
which implies that $\ds v(x_1)\leq H(u(x_1))\leq \max_{0\leq u\leq u^*} H(u):=H^*$ as $0\leq u(x_1)\leq u^*$.
\end{proof}

Denote the set of positive solutions of (\ref{8}) by
\begin{equation*}
  \Sigma=\{(d,u,v)\in \R\times X^2: d>0, u>0, v>0, W(d,u,v)=(0,0)\},
\end{equation*}
where $W$ is defined in \eqref{W}.
We have  the following result on the global bifurcation of positive solutions of \eqref{8} when $a>0$.
\begin{theorem}\label{thm3}
Suppose that the conditions {\bf (A1')}, {\bf (A2)}-{\bf (A4)} hold and $a>0$. Then the following results are true:
  \begin{enumerate}
 \item (\ref{8}) has no positive solution when $d>d^{**}:=K_0/\lambda_1$;

 \item there exists a connected component $\Sigma_1$ of $\Sigma$ such that $\Gamma_1\subseteq\Sigma_1$, the projection $P_d\Sigma_1$ of $\Sigma_1$ into the $d-$component satisfies $P_d\Sigma_1=(0,d_0)$ for some $d_0\in[d^*,d^{**})$, and for every $(d,u,v)\in\Sigma_1$, $||u||_{\infty}+||v||_{\infty}\leq C$ for some $C>0$ independent of $d$.
\end{enumerate}

\end{theorem}
\begin{proof}
  1.  Suppose that $(d,u,v)$ is a positive solution of \eqref{8}. Multiplying the first equation of (\ref{8}) by $\phi_1$ and integrating on $\Omega$, we obtain
  \begin{equation*}
  \begin{split}
        \lambda_1d\int_{\Omega}u(x)\phi_1(x)dx=&-d\int_{\Omega}\Delta u(x)\phi_1(x)dx=\int_{\Omega}\phi_1(x)F(u(x),v(x))dx\\
        \leq& \int_{\Omega}\phi_1(x)F_1(u(x))u(x)dx \leq K_0\int_{\Omega}u(x)\phi_1(x)dx.
          \end{split}
  \end{equation*}
  That is, $(d\lambda_1-K_0)\ds\int_{\Omega}u(x)\phi_1(x)dx\leq0$. Thus, the system (\ref{8}) have no positive solution if $d>K_0/\la_1$.


2. According to Krasnoselskii-Rabinowitz global bifurcation theorem (see \cite{Rabinowitz1971,shi2009global}), a connected component $\Sigma_1$ of $\Sigma$  that contains $\Gamma_1$ (defined in Theorem \ref{thm:3.1}) satisfies one of the following: (i) $\Sigma_1$ is unbounded; or (ii) $\overline{\Sigma_1}$ contains $(\tilde{d},0,0)$, where $(\tilde{d},0,0)$ is another bifurcation point from $\Gamma_0$ such that $\tilde{d}>0$ (the line of trivial solutions); or (iii) $\overline{\Sigma_1}$ contains $(\hat{d},\hat{u},\hat{v})$ which is on the boundary $\partial S$ of $S=\{(d,u,v)\in \R\times X^2: d>0, u>0, v>0\}$.

From Theorem \ref{thm:3.1}, we know the case (ii) cannot occur as $d=d^*$ is the only bifurcation point for positive solutions of \eqref{8}. From Lemma \ref{lem5}, any positive solution $(u,v)$ of \eqref{8} satisfies $||u||_{\infty}+||v||_{\infty}\leq u^*+H^*$ which is independent of $d$; and from part 1 of Theorem \ref{thm3}, any solution $(d,u,v)$ of \eqref{8} must satisfy $0\le d\le d^{**}$. Hence the alternative (i) cannot occur either. Therefore (iii) occurs, and $\overline{\Sigma_1}$ contains $(\hat{d},\hat{u},\hat{v})$ which is on  $\partial S$. From the strong maximum principle, if $\hat{u}(x)=0$ for some $x\in \Om$, then $\hat{u}(x)\equiv 0$ for $x\in \Omega$. If $\hat{u}(x)\equiv 0$, it is easy to see $\hat{v}\equiv 0$. If we also have $\hat{d}>0$ then this returns to the case (ii). Thus we must have $\hat{d}=0$.  This shows that $P_d\Sigma_1\supset (0,d^*)$. Let $d_0=\sup \{d>0: (d,u,v)\in \Sigma_1\}$. Then $d_0\geq d^*$, and from part 1 of Theorem \ref{thm3}, we also have $d_0<d^{**}$. This completes the proof.
\end{proof}

\subsection{Case 2: $a=F_u(0,0)<0$.}

In this subsection we further assume  the following condition holds:
\begin{enumerate}
\item[\bf(A5)] There exist positive constants $K_1,~K_2,~K_3$ and a continuous function $F_2:\bar{\R}_+\rightarrow \R$
   such that $F(u,v)\leq -K_1u+F_2(v)$ for $(u,v)\in \bar{\R}_+\times\bar{\R}_+$,  $F_2(v)\leq K_2 v$ for $v\in\bar{\R}_+$, and $H(u)\leq K_3 u$ for $u\in\bar{\R}_+$.
   \item[\bf(A6a)]  There exists a positive constant $K_4$ such that  $F_2(v)\leq K_4$ for $v\in\bar{\R}_+$; or
   \item[\bf(A6b)] There exists a positive constants $K_5$ such that  $H(u)\leq K_5$ for $u\in\bar{\R}_+$.
\end{enumerate}
We remark that {\bf (A5)} implies that
\begin{equation}\label{H6}
  a=F_u(0,0)\leq -K_1, \;\; b=F_v(0,0)\leq K_2, \;\; k=H'(0)\leq K_3.
\end{equation}

Similar to Lemma \ref{lem5} we have the following \textit{a priori} estimates under {\bf (A5a)} or {\bf (A5b)}.
\begin{lemma}\label{lem4}
   Suppose the conditions {\bf(A1')}, {\bf (A2)}, {\bf(A3)}, {\bf (A5)} hold  and $a<0$, $(u,v)$ is a nonnegative solution of (\ref{8}).
   \begin{enumerate}
     \item When {\bf (A6a)} is also satisfied, then
  \begin{equation}
    0\leq u(x)\leq\frac{K_4}{K_1},\;\; 0\leq v(x)\leq \max_{0\leq u\leq K_4/K_1} H(u):=H^{**}.
  \end{equation}
     \item When {\bf (A6b)} is also satisfied, then
  \begin{equation}
    0\leq u(x)\leq\frac{1}{K_1}\max_{0\leq v\leq K_5} F_2(v):=H^{***},\;\; 0\leq v(x)\leq K_5.
  \end{equation}
   \end{enumerate}
  \end{lemma}

\begin{proof}
If $u(x)\equiv 0$, then $v(x)\equiv 0$ and the result is obviously true. Thus we assume that $u(x)>0$ for $x\in \Omega$ from the maximum principle. First we assume that {\bf (A6a)} is  satisfied. Let $x_0\in\Omega$ such that $u(x_0)=\max\limits_{x\in\bar{\Omega}}u(x)>0$. Then  from the maximum principle,  the first equation of (\ref{8}) and {\bf (A5)}, we have that
\begin{equation}\label{H1}
 0\leq-d\Delta u(x_0)=F(u(x_0),v(x_0))\leq-K_1u(x_0)+F_2(v(x_0)),
\end{equation}
which together with {\bf (A6a)} implies that $u(x_0)\leq \ds \frac{F_2(v(x_0))}{K_1}\leq \frac{K_4}{K_1}$ and hence $0<u(x)\leq \ds\frac{K_4}{K_1}$ for  $x\in \Omega$ from the strong maximum principle.

Since $u(x)>0$,  $v(x)=(-d\Delta+\tau^{-1})^{-1}(\tau^{-1}H(u))>0$ for $x\in \Omega$. Let $x_1\in\Omega$ such that $v(x_1)=\max\limits_{x\in\bar{\Omega}}v(x)>0$. Then  from the maximum principle and  the second equation of (\ref{8}), we have that
\begin{equation}\label{H2}
0\leq-d\Delta v(x_1)=\frac{1}{\tau}[H(u(x_1))-v(x_1)],
\end{equation}
which implies that for any $x\in \Omega$, $\ds v(x)\leq v(x_1)\leq H(u(x_1))\leq \max_{0\leq u\leq K_4/K_1} H(u):=H^{**}$ as $0\leq u(x_1)\leq \ds\frac{K_4}{K_1}$.

Next we  assume that {\bf (A6b)} is satisfied. Let $x_0$ and $x_1$ be the same definition as above. Then from \eqref{H2} and {\bf (A6b)}, we have $\ds v(x)\leq v(x_1)\leq H(u(x_1))\leq K_5$ for any $x\in \Omega$, and from \eqref{H1}, we have  $u(x)\leq u(x_0)\leq \ds \frac{F_2(v(x_0))}{K_1}\leq \frac{1}{K_1}\max_{0\leq v\leq K_5} F_2(v):=H^{***}$.
%
%
%
%
%
\end{proof}

Now we have  the following results on the global bifurcation of positive solutions of \eqref{8}.
\begin{theorem}\label{thm5}
Suppose that the conditions {\bf(A1')}, {\bf (A2)}, {\bf(A3)}, {\bf(A5)}, {\bf (A6a)} or {\bf (A6b)} hold and $a<0$. Then the following results are true:
  \begin{enumerate}
 \item (\ref{8}) has no positive solution when $d>d^{***}$ which is defined as
  \begin{equation}\label{H7}
     d^{***}(\tau)=\frac{1}{2\lambda_1\tau}(-K_1\tau-1+\sqrt{(-K_1\tau+1)^2+4\tau K_2 K_3}),
  \end{equation}
  and $K_1,~K_2,~K_3$ are defined in {\bf (A5)};

 \item there exists a connected component $\Sigma_1$ of $\Sigma$ such that $\Gamma_1\subseteq\Sigma_1$, the projection $P_d\Sigma_1$ of $\Sigma_1$ into the $d-$component satisfies $P_d\Sigma_1=(0,d_0)$ for some $d_0\in[d^*,d^{***})$, and for every $(d,u,v)\in\Sigma_1$, $||u||_{\infty}+||v||_{\infty}\leq C$ for some $C>0$ independent of $d$.
\end{enumerate}
 \end{theorem}

\begin{proof}
1.  Assume that $(d,u,v)$ is a positive solution of (\ref{8}). Multiplying the second equation of (\ref{8}) by $\phi_1(x)$ and integrating on $\Omega$, using {\bf (A5)} we have
  \begin{equation*}
  \begin{split}
   &\lambda_1d\int_{\Omega}v(x)\phi_1(x)dx=-d\int_{\Omega}\Delta v(x)\phi_1(x)dx
   =\frac{1}{\tau}\int_{\Omega}H(u(x))\phi_1(x)dx-\frac{1}{\tau}\int_{\Omega}v(x)\phi_1(x)dx\\
   \leq&\frac{K_3}{\tau}\int_{\Omega}u(x)\phi_1(x)dx-\frac{1}{\tau}\int_{\Omega}v(x)\phi_1(x)dx,
   \end{split}
  \end{equation*}
  which implies
  \begin{equation}\label{H3}
    \left(\lambda_1d+\frac{1}{\tau}\right)\int_{\Omega}v(x)\phi_1(x)dx\leq\frac{K_3}{\tau}\int_{\Omega}u(x)\phi_1(x)dx.
  \end{equation}
Similarly  multiplying the first equation of (\ref{8}) by $\phi_1(x)$ and integrating on $\Omega$, using {\bf (A5)} we have
  \begin{equation*}
  \begin{split}
       & \lambda_1d\int_{\Omega}u(x)\phi_1(x)dx=-d\int_{\Omega}\Delta u(x)\phi_1(x)dx=\int_{\Omega}F(u(x),v(x))\phi_1(x)dx\\
        \leq& -K_1\int_{\Omega}u(x)\phi_1(x)dx +K_2\int_{\Omega}v(x)\phi_1(x)dx,
          \end{split}
  \end{equation*}
   which implies
  \begin{equation}\label{H4}
    \left(\lambda_1d+K_1\right)\int_{\Omega}u(x)\phi_1(x)dx\leq K_2\int_{\Omega}v(x)\phi_1(x)dx.
  \end{equation}
  Combining \eqref{H3} and \eqref{H4}, we obtain that
  \begin{equation}\label{H8}
    (\lambda_1d+K_1)\left(\lambda_1d+\frac{1}{\tau}\right)\leq \frac{K_2 K_3}{\tau}.
  \end{equation}
  It is easy to calculate that \eqref{H8} holds when $0<d\leq d^{***}$, since $K_2 K_3-K_1\geq bk+a>0$ from {\bf (A3)} and \eqref{H6}.
   Therefore, system (\ref{8}) have no positive solution if $d>d^{***}$.

2. According to Krasnoselskii-Rabinowitz global bifurcation theorem (see \cite{Rabinowitz1971,shi2009global}), a connected component $\Sigma_1$ of $\Sigma$   that contains $\Gamma_1$ (defined in Theorem \ref{thm:3.1}) satisfies one of the following: (i) $\Sigma_1$ is unbounded; or (ii) $\Sigma_1$ contains $(\tilde{d},0,0)$, where $(\tilde{d},0,0)$ is another bifurcation point from $\Sigma_0$; or (iii) $\overline {\Sigma_1}$ contains $(\hat{d},\hat{u},\hat{v})$, which is on the boundary $\partial S$ of $S=\{(d,u,v)\in \R\times X^2: d>0, u>0, v>0\}$.

From Theorem \ref{thm:3.1}, we know the case (ii) cannot occur as $d=d^*$ is the only bifurcation point for positive solutions of \eqref{8}. From Lemma \ref{lem4}, any positive solution $(u,v)$ of \eqref{8} satisfies $||u||_{\infty}+||v||_{\infty}\leq \ds\frac{K_4}{K_1}+H^{**}$ or $H^{***}+K_5$, which is independent of $d$; and from part 1 of Theorem \ref{thm5}, any solution $(d,u,v)$ of \eqref{8} must satisfy $0\leq d\leq d^{***}$. Hence the alternative (i) cannot occur either. Therefore (iii) occurs, and $\overline{\Sigma_1}$ contains $(\hat{d},\hat{u},\hat{v})$ which is on  $\partial S$. Similar to the proof of Theorem \ref{thm3} we must have $\hat{d}=0$ thus $P_d\Sigma_1\supset (0,d^*)$. Let $d_0=\sup \{d>0: (d,u,v)\in \Sigma_1\}$. Then $d_0\geq d_*$, and from part 1 of Theorem \ref{thm5}, we also have $d_0<d^{***}$. This completes the proof.
\end{proof}

\section{Uniqueness of the steady state}

In Sections 3 and 5, the existence of a positive steady state solution of \eqref{1} for all small diffusion coefficient case $d\in (0,d_0)$ has been proved under proper conditions on the nonlinear functions $F$ and $H$. In general the positive steady state solution is not necessarily unique for all $d\in (0,d_0)$, except near the bifurcation point $d=d^*$. Here we show that when the spatial domain is one-dimensional and the nonlinearity is in a more special form, the positive steady state solution of \eqref{1} is unique for all $d\in (0,d_0)$ due to its ``consumer-resource" type structure. 

This section focuses on the one-dimensional steady state problem with the nonlinearity being in a form of $F(u,v)=uf(u,v)$:
\begin{equation}\label{52}
  \begin{cases}
    -du''(x)=u(x) f(u(x),v(x)),&x\in(0,L),\\
   \ds -dv''(x)=\frac{1}{\tau}(H(u(x))-v(x)),&x\in(0,L),\\
    u(0)=u(L)=v(0)=v(L)=0,
  \end{cases}
\end{equation}
where $L>0$ and $':=\ds\frac{d}{dx}$. The linearized equation  at a positive solution $(u_d(x),v_d(x))$ of  (\ref{52})  can be written as
\begin{equation}
  \label{53}
  \begin{cases}
    -d\phi''-[u_df_u(u_d,v_d)+f(u_d,v_d)]\phi=u_df_v(u_d,v_d)\psi,&x\in(0,L),\\
    \ds -d\psi''+\frac{1}{\tau}\psi=\frac{1}{\tau}H'(u_d)\phi,&x\in(0,L),\\
    \phi(0)=\phi(L)=\psi(0)=\psi(L)=0.
  \end{cases}
  \end{equation}
The coexistence state $(u_d(x),v_d(x))$ of  (\ref{52}) is non-degenerate if the only solution of \eqref{53}  is $(\phi,\psi)=(0,0)$. The key of establishing the uniqueness of positive solution of \eqref{52} is the following non-degeneracy property of positive solution.
\begin{proposition}\label{pro:non}
Suppose that the conditions {\bf(A1')}, {\bf (A2)} hold for $F(u,v)=uf(u,v)$ and $f(u,v)$, $H(u)$ also satisfy
  \begin{enumerate}
\item[\bf(A7)] $f_u(u,v)< 0$ and $f_v(u,v)<0$ for $(u,v)\in \R^+\times\R^+$, and $H'(u)>0$ for $u\in \R^+$.
\end{enumerate}
  If $(u_d(x),v_d(x))$ is a positive solution of (\ref{52}),  then $(u_d(x),v_d(x))$ is non-degenerate.
\end{proposition}
\begin{proof}
  Since $(u_d(x),v_d(x))$ solves (\ref{52}), the following equalities hold:
  \begin{equation*}
    \begin{cases}
      \ds \left(-d\frac{d^2}{dx^2}-f(u_d,v_d)\right)u_d=0,\\
      \ds \left(-d\frac{d^2}{dx^2}+\frac{1}{\tau}-\frac{1}{\tau}\frac{H(u_d)}{v_d}\right)v_d=0,\\
      u_d(0)=u_d(L)=v_d(0)=v_d(L)=0.
    \end{cases}
  \end{equation*}
Since $(u_d,v_d)$ is positive, it follows from the Krein-Rutman Theorem,
\begin{equation}\label{54}
\rho_1\left(-d\frac{d^2}{dx^2}-f(u_d,v_d)\right)=\rho_1\left(-d\frac{d^2}{dx^2}+\frac{1}{\tau}-\frac{1}{\tau}\frac{H(u_d)}{v_d}\right)=0,
\end{equation}
where $\rho_1(L)$ is the principal eigenvalue corresponding to the operator $L$. Clearly, the linearized equation (\ref{53}) can be rewritten as
  \begin{equation}\label{56}
    \begin{cases}
      \ds L_1\phi\triangleq\left(-d\frac{d^2}{dx^2}-u_df_u(u_d,v_d)-f(u_d,v_d)\right)\phi=u_df_v(u_d,v_d)\psi,\\
     \ds L_2\psi\triangleq\left(-d\frac{d^2}{dx^2}+\frac{1}{\tau}\right)\psi=\frac{1}{\tau}H'(u_d)\phi,\\
      \phi(0)=\phi(L)=\psi(0)=\psi(L)=0.
    \end{cases}
  \end{equation}
By the monotonicity of principal eigenvalue $\rho_1(\cdot)$, (\ref{54}) and {\bf (A7)}, we have
  \begin{equation}\label{55}
    \begin{split}
      & \rho_1(L_1)=\rho_1\left(-d\frac{d^2}{dx^2}-u_df_u(u_d,v_d)-f(u_d,v_d)\right)>\rho_1\left(-d\frac{d^2}{dx^2}-f(u_d,v_d)\right)=0,\\
      &\rho_1(L_2)=\rho_1\left(-d\frac{d^2}{dx^2}+\frac{1}{\tau}\right) >\rho_1\left(-d\frac{d^2}{dx^2}+\frac{1}{\tau}-\frac{1}{\tau}\frac{H(u_d)}{v_d}\right)=0.
    \end{split}
  \end{equation}
From (\ref{55}), we know that all eigenvalues of the operators $L_1$ and $L_2$ are positive, and they have the inverse operators $L_1^{-1}$ and $L_2^{-1}$ respectively, which are compact, strictly order-preserving with respect to the usual cone of positive functions.

We prove that the only solution of \eqref{56} is $(\phi,\psi)=(0,0)$ by contradiction.
Suppose (\ref{56}) has a nontrivial solution  $(\phi,\psi)\neq(0,0)$. From (\ref{56}), we have
\begin{equation}
  \label{57}
  \phi=L_1^{-1}\left(u_d f_v(u_d,v_d)L_2^{-1}\left(\frac{1}{\tau}H'(u_d)\phi\right)\right).
\end{equation}
Since $f_v(u,v)<0$ and $H'(u)>0$, and the right hand side of (\ref{57}) determines a compact, strongly order-preserving operator. Thus, $\phi$ must change sign in $(0,L)$, and consequently $\psi$ must change sign in $(0,L)$. Now we can follow the argument in \cite[Lemma 3.1]{LG1993} or \cite[Lemma 5.2]{CasalEilbeck1994} to show that $\phi(x)= \psi(x)\equiv 0$ for $x\in (0,L)$.
\end{proof}

Now we can prove the uniqueness of positive steady state and exact global bifurcation when $a=F_u(0,0)>0$ and $\Omega=(0,L)$.

\begin{theorem}\label{thm4}
Suppose that the conditions {\bf (A1')}, {\bf (A2)}-{\bf(A4)}, {\bf (A7)} hold and $a>0$. Then \eqref{52} has a unique positive solution $(u_d(x),v_d(x))$ which is non-degenerate when $0<d<d^*$, and it has no positive solution when $d\geq d^*$. Moreover all positive solutions of \eqref{52} are on a smooth curve $\Sigma_1=\{(d,u_d(x),v_d(x)):0<d<d^*\}$.
\end{theorem}
\begin{proof}
From  {\bf (A7)}, $f(u,v)<f(0,0)$ for any $(u,v)\in \R^+\times\R^+$. If $(u,v)$ is a positive solution of \eqref{52}, by integrating
\begin{equation*}
  -duu''=u^2f(u,v)<u^2 f(0,0), \;\; u(0)=u(L)=0,
\end{equation*}
we obtain
\begin{equation*}
  d\int_0^L [u'(x)]^2 dx<f(0,0)\int_0^L u^2(x) dx\leq \frac{f(0,0)}{\lambda_1}\int_0^L [u'(x)]^2 dx.
\end{equation*}
This implies that $d\leq \ds \frac{f(0,0)}{\lambda_1}=\frac{F_u(0,0)}{\lambda_1}=\frac{a}{\lambda_1}=d^*$. Hence \eqref{52} has no positive solution when $d\geq d^*$. On the other hand, the existence of  positive solution of \eqref{52} has been shown in Theorem \ref{thm3}. In particular, for $d\in (d^*-\epsilon,d^*)$, \eqref{52} has a positive solution $(u_d,v_d)$ so that $\ds\lim_{d\to (d^*)^{-}}(u_d,v_d)=(0,0)$, and these solutions are on a curve $\Gamma_1=\{(d(s),u(s),v(s)):s\in (0,\delta_1)\}$. Note now the direction of the curve $\Gamma_1$ is given by
\begin{equation*}
    d'(0)=(f_u(0,0)+f_v(0,0)M)\frac{\ds\int_0^L\phi^3_1(x)dx}{\lambda_1\ds\int_0^L\phi^2_1(x)dx}<0,
\end{equation*}
as $b=0$, $p=2f_u(0,0)<0$, $q=f_v(0,0)<0$, $r=0$, and $M=k/(a\tau+1)>0$. From Theorem \ref{thm3}, $\Gamma_1\subset \Sigma_1$ which is a connected component of the set of positive solution $\Sigma$ of \eqref{52}, and $P_d\Sigma_1= (0,d^*)$. From Proposition \ref{pro:non}, any positive solution on $\Sigma_1$ is non-degenerate, so $\Sigma_1$ is locally a smooth curve at any $(d,u_d,v_d)\in\Sigma_1$ hence $\Sigma_1$ can be globally parameterized by $d\in (0,d^*)$. Suppose that for some $d\in (0,d^*)$, there is another positive solution $(d,\hat{u_d},\hat{v_d})$ not on $\Sigma_1$, then using the same argument and Proposition \ref{pro:non}, we can show that $(d,\hat{u_d},\hat{v_d})$ is on another connected component $\Sigma_2$ of $\Sigma$, and $\Sigma_2$ is also globally a smooth curve. We also have $P_d\Sigma_2= (0,d^*)$ as $d=d^*$ is the only bifurcation point for positive solutions of \eqref{52}. But the local bifurcation result in Theorem \ref{thm:3.1} shows that near $d=d^*$ the positive solution is unique for \eqref{52}, which contradicts with the existence of two solutions $(d,u_d,v_d)$ and $(d,\hat{u_d},\hat{v_d})$. So such a second component $\Sigma_2$ cannot exist, and  all positive solutions of \eqref{52} are on the smooth curve $\Sigma_1=\{(d,u_d(x),v_d(x)):0<d<d^*\}$. In particular, the positive solution of \eqref{52} is unique for $0<d<d^*$.
\end{proof}

\section{Applications}

In this section, we apply the previous main results obtained in Sections 2-6 to the following logistic type and Nicholson's blowfly type models with nonlocal delay.

\subsection{Logistic type models}
We consider a modified Hutchinson's equation with diffusion and nonlocal delay:
\begin{equation}
  \label{66}
  \begin{cases}
     u_t(x,t)=d\Delta u(x,t)+\kappa u(x,t)(1-Au(x,t)-B(g\ast\ast u)(x,t)),&~x\in\Omega,~t>0,\\
     u(x,t)=0,&~x\in\partial\Omega,~t>0,
  \end{cases}
\end{equation}
where $\kappa>0$ is the maximum growth rate per capita, the parameters $A,B>0$ denote the portions of instantaneous and previous dependence of the growth rate, respectively.

Then by Section 2, the system (\ref{66}) is equivalent to the following system:
\begin{equation}\label{67}
  \begin{cases}
    u_t(x,t)=d\Delta u(x,t)+\kappa u(x,t)(1-Au(x,t)-Bv(x,t)),&~x\in\Omega,~t>0,\\
    \ds v_t(x,t)=d\Delta v(x,t)+\frac{1}{\tau}(u(x,t)-v(x,t)),&~x\in\Omega,~t>0,\\
     u(x,t)=v(x,t)=0,&~x\in\partial\Omega,~t>0.
  \end{cases}
\end{equation}
Let $F(u,v)=\kappa u(1-Au-Bv)$ and $H(u)=u$. It is easy to compute that
\begin{equation*}
  \begin{split}
    &a=F'_u(0,0)=\kappa>0,~b=F'_v(0,0)=0,~k=H'(0)=1,\\
    &p=F_{uu}(0,0)=-2A\kappa,~q=F_{uv}(0,0)=-B\kappa,~r=F_{vv}(0,0)=0,~l=H''(0)=0.
  \end{split}
\end{equation*}
By Theorem \ref{thm:3.1}(2), Theorem \ref{50}(1), Theorem \ref{thm3} and Theorem \ref{thm4}, we obtain the following results:
\begin{proposition}\label{pro1}
  Suppose that $A,B,\kappa,\tau>0$, and denote $d^*=\ds\frac{\kappa}{\lambda_1}$.
 \begin{enumerate}
\item System (\ref{67}) has at least one  positive steady state solution $(u_d(x),v_d(x))$ for any $d\in(0,d^*)$ and has no positive steady state solution for $d>d^*$; the positive steady state $(u_d(x),v_d(x))$ satisfies $u_d(x),v_d(x)\leq 1/A$; there is a connected component $\Sigma_1$ of the set of positive steady state solutions of \eqref{66} such that $P_d\Sigma_1=(0,d^*)$; near $d=d^*$, $\Sigma_1$ is a smooth curve $\{(d,u_d(\cdot),v_d(\cdot)): d^*-\epsilon<d<d^*\}$ such that $\ds\lim_{d\to (d^*)^-}u_d(\cdot)=\lim_{d\to (d^*)^-}v_d(\cdot)=0$, and $(u_d(\cdot),v_d(\cdot))$ is locally asymptotically stable for $d\in(d^*-\epsilon,d^*)$.
\item When $d>d^*$, the trivial steady state solution $u=0$ is globally asymptotically stable for  (\ref{67}); and when $0<d<d^*$, $u=0$ is unstable.
\item For $\Omega=(0,L)\subset \R^1$, the positive steady state solution $u_d(x)$ of system (\ref{67}) is unique and non-denegerate for $d\in(0,d^*)$.
     \end{enumerate}
\end{proposition}
\begin{proof}
 1.  It is easy to verify that the conditions {\bf(A1)}-{\bf(A3)} and {\bf (A1')} hold and according to (\ref{first}), we have
  \begin{equation}\label{d0}
    d'(0)=-\frac{\kappa(A+BM)\int_{\Omega}\phi^3_1(x)dx}{\lambda_1\int_{\Omega}\phi^2_1(x)dx}<0,
  \end{equation}
  where $M=1/(\tau+1)>0$. Then the local bifurcation and stability of positive steady state solutions of \eqref{67} follows from Theorem \ref{thm:3.1} part 2 and Theorem \ref{50} part 1. Define $F_1(u)=\kappa(1-Au)$ which satisfies $F(u,v)\leq F_1(u)u$, $F_1(1/A)=0$ and $0<F_1(u)<\kappa$ for $u\in(0,1/A)$ and $F_1(u)<0$ for $u>1/A$. That is, {\bf (A4)} holds. Then by Theorem \ref{thm3},  (\ref{67}) has at least one  positive steady state solution $u_d(x)$ for any $d\in(0,d^*)$, and \eqref{66} has no positive steady state solution for $d>d^*$ following the proof of Theorem \ref{thm4}. Moreover from Lemma \ref{lem5}, any positive steady state satisfies $u_d(x),v_d(x)\leq 1/A$.

  2. When $d>d^*$, we have $u_t\leq d\Delta u+\kappa u(1-Au)$, then the global stability of $u=0$ follows from well-known results for the logistic reaction-diffusion model (see for example \cite{CC2003}). When $d<d^*$, it is standard to show that $u=0$ is unstable.

  3. Let $f(u,v)=\kappa (1-Au-Bv)$. Clearly $f_u(u,v)=-\kappa A<0$ and $f_v(u,v)=-\kappa B<0$ so the condition {\bf (A7)} holds. By Theorem \ref{thm4}, the system (\ref{67}) has a unique positive solution $(u_d(x),v_d(x))$ for $d\in(0,d^*)$ when $\Om=(0,L)$.
\end{proof}

As a numerical  example, we consider (\ref{66}) with $\kappa=1,~A=0.5,~B=0.4,~\tau=0.5,~\Omega=(0,\pi)$ and choose the initial condition $\eta(x,t)=0.1\sin{x},~t\in(-\infty,0)$. When $d=1.05>d^*=\kappa/\lambda_1=1$, the zero solution  is globally asymptotically stable from Proposition \ref{pro1}, illustrated in Fig. \ref{Figure7} (A). On the other hand when $d=0.5<d^*=1$, the zero solution loses its stability and the unique positive steady state solution appears to be asymptotically stable as shown in Fig. \ref{Figure7}  (B).

\begin{figure}[h]
\includegraphics[width=0.4\textwidth]{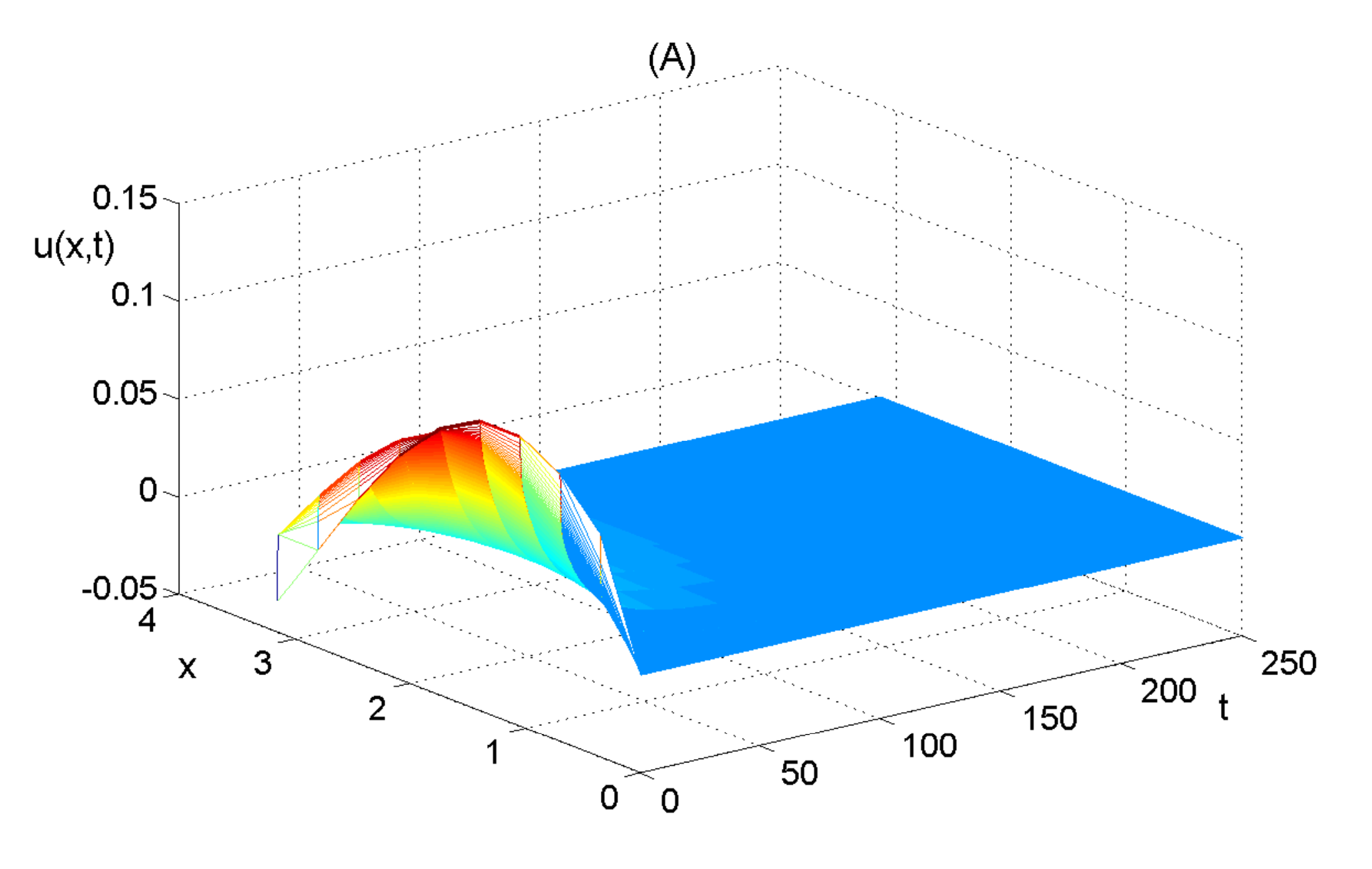}\hspace{0.1in}\includegraphics[width=0.4\textwidth]{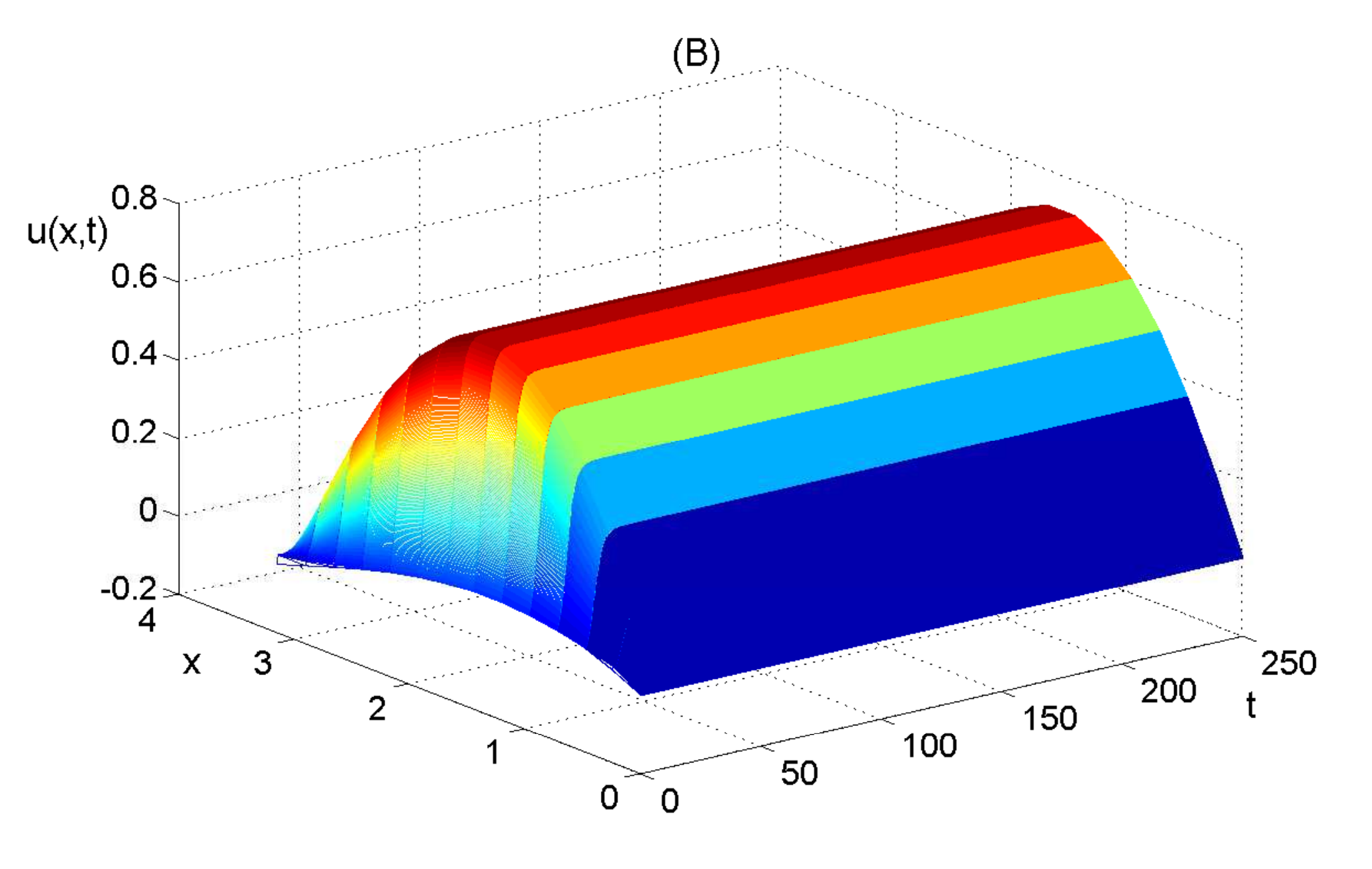}
\caption{Numerical simulations of (\ref{66}) with $\kappa=1,~A=0.5,~B=0.4,~\tau=0.5$ and $\eta(x,t)=0.1\sin{x},~t\in(-\infty,0)$. (A): $d=1.05$, the solution converges to the trivial solution $u\equiv 0$; (B): $d=0.5$, the solution converges to a positive steady state solution.}\label{Figure7}
\end{figure}

It is an interesting open question whether the uniqueness of positive solution of \eqref{66} holds for the general domain $\Omega\in \R^n$ with $n\ge 2$, and the local/global stability of the positive solution of \eqref{66} is also not known even in the case of $n=1$. Note that the non-degeneracy shown in Proposition \ref{pro:non} rules out the zero eigenvalue of linearized equation, but Hopf bifurcation can still occur to destabilize the positive steady state. When the boundary condition of \eqref{66} is Neumann one, it is known that the positive steady state is unique and is constant in space \cite{Ni2018}. We also remark that the bifurcation at $d=d^*$ is supercritical so that the bifurcating positive steady state solutions are stable ones. A subcritical bifurcation is possible in the following variant of \eqref{66}:
\begin{equation}
  \label{66vv}
  \begin{cases}
     u_t(x,t)=d\Delta u(x,t)+\kappa u(x,t)(1+Au(x,t)-B(g\ast\ast u)(x,t)-Cu^2(x,t)),&~x\in\Omega,~t>0,\\
     u(x,t)=0,&~x\in\partial\Omega,~t>0,
  \end{cases}
\end{equation}
where $\kappa,A,B,C>0$. For \eqref{66vv}, results similar to the ones in Proposition \ref{pro1} can be proved and the equation \eqref{d0} becomes
\begin{equation}\label{d000}
    d'(0)=-\frac{\kappa(-A+BM)\int_{\Omega}\phi^3_1(x)dx}{\lambda_1\int_{\Omega}\phi^2_1(x)dx}.
  \end{equation}
So the bifurcation is subcritical if $-A+BM<0$, and system \eqref{66vv} has multiple positive steady state solutions for $d\in (d^*,d^*+\epsilon)$.

Another example with similar structure is the food-limited population model considered in \cite{Gourley2002}:
\begin{equation}
  \label{88vv}
  \begin{cases}
     \ds u_t(x,t)=d\Delta u(x,t)+\kappa u(x,t)\frac{1-Au(x,t)-B(g\ast\ast u)(x,t)}{1+cAu(x,t)+cB(g\ast\ast u)(x,t)},&~x\in\Omega,~t>0,\\
     u(x,t)=0,&~x\in\partial\Omega,~t>0,
  \end{cases}
\end{equation}
where $\kappa,A,B,c>0$. Note that when $c=0$, \eqref{88vv} is reduced to \eqref{66}. Indeed all results in Proposition \ref{pro1} also hold for \eqref{88vv} as well.

 \subsection{Nicholson's blowfly type models}

Consider the diffusive Nicholson's Blowflies equation with nonlocal delay as follows \cite{LiRuanWang2007}:
\begin{equation}\label{70}
  \begin{cases}
    u_t(x,t)=d\Delta u(x,t)-\chi u(x,t)+\vartheta (g\ast\ast u)(x,t)e^{-\nu (g\ast\ast u)(x,t)},&~x\in\Omega,~t>0,\\
    u(x,t)=0,&~x\in\partial\Omega,~t>0.
  \end{cases}
\end{equation}
Here $\chi$ is the per capita daily adult death rate, $\vartheta$ is the maximum per capita daily egg production rate, $1/\nu$ is the size at which the blowfly population reproduces at its maximum rate, and $\tau$ is the generation time.
From the equivalence relation shown in Section 2, the system (\ref{70}) is equivalent to the reaction-diffusion system:
\begin{equation}\label{71}
  \begin{cases}
    u_t(x,t)=d\Delta u(x,t)-\chi u(x,t)+\vartheta v(x,t)e^{-\nu  v(x,t)},&~x\in\Omega,~t>0,\\
    \ds v_t(x,t)=d\Delta v(x,t)+\frac{1}{\tau}(u(x,t)-v(x,t)),&~x\in\Omega,~t>0,\\
    u(x,t)=v(x,t)=0,&~x\in\partial\Omega,~t>0,
  \end{cases}
\end{equation}
whose steady steady state solutions satisfy the following equations:
\begin{equation}\label{72}
  \begin{cases}
    -d\Delta u(x)=-\chi u(x)+\vartheta v(x)e^{-\nu  v(x)},&~x\in\Omega,\\
    \ds -d\Delta v(x)=\frac{1}{\tau}(u(x)-v(x)),&~x\in\Omega,\\
    u(x)=v(x)=0,&~x\in\partial\Omega.
  \end{cases}
\end{equation}
 Let $F(u,v)=-\chi u+\vartheta ve^{-\nu  v},~H(u)=u$. It is easy to compute that from (\ref{41}),
 \begin{equation*}
  \begin{split}
    &a=F_u(0,0)=-\chi<0,~b=F_v(0,0)=\vartheta,~k=H'(0)=1,~H(0)=0,\\
    &p=F_{uu}(0,0)=0,~q=F_{uv}(0,0)=0,~r=F_{vv}(0,0)=-2\vartheta^2<0,~l=H''(0)=0.
  \end{split}
\end{equation*}
By Theorem \ref{thm:3.1}(2), Theorem \ref{50}(1) and Theorem \ref{thm5}, we obtain the following results:
\begin{proposition}
  \label{pro2}
  Suppose that $\chi,\nu,\tau>0$ and $\vartheta>\chi$, and denote
  \begin{equation}\label{dstar}
    d^*=\ds\frac{1}{2\lambda_1\tau}(-\chi\tau-1+\sqrt{(-\chi\tau+1)^2+4\vartheta\tau}).
\end{equation}
 \begin{enumerate}
    \item System \eqref{71} has at least one positive steady state solution $(u_d(x),v_d(x))$ for any $d\in (0,d^*)$ and has no positive steady state solution for $d>d^*$; there is a connected component $\Sigma_1$ of the set of positive steady state solutions of \eqref{71} such that $P_d\Sigma_1=(0,d^*)$; near $d=d^*$, $\Sigma_1$ is a smooth curve $\{(d,u_d(\cdot),v_d(\cdot)): d^*-\epsilon<d<d^*\}$ such that $\ds\lim_{d\to (d^*)^-}u_d(\cdot)=\lim_{d\to (d^*)^-}v_d(\cdot)=0$, and $(u_d(\cdot),v_d(\cdot))$ is locally asymptotically stable for $d\in(d^*-\epsilon,d^*)$.
\item The positive steady state $(u_d(x),v_d(x))$ satisfies $u_d(x),v_d(x)\leq \vartheta/(\nu \chi e)$ for all $0<d<d^*$. Moreover if  $(u_d(x),v_d(x))$  satisfies $v_d(x)\leq 1/\nu$, then it is locally asymptotically stable.
      \end{enumerate}
\end{proposition}
\begin{proof}
1. It is easy to verify that the conditions {(\bf{A1})}-{(\bf{A3})}, {\bf (A1')} hold and according to (\ref{first}), we have
  \begin{equation*}
    d'(0)=-\frac{-\vartheta^2M^2\int_{\Omega}\phi^3_1(x)dx}{\lambda_1(1+M^2\vartheta\tau)\int_{\Omega}\phi^2_1(x)dx}<0,
  \end{equation*}
  where
  \begin{equation*}
    M=\frac{2}{-\chi\tau+1+sqrt{(-\chi\tau+1)^2+4\vartheta \tau}}.
  \end{equation*}
    Then the local bifurcation and stability of positive steady state solutions of \eqref{71} follows from Theorem \ref{thm:3.1} part 2 and Theorem \ref{50} part 1. Let $K_1=\chi$, $K_2=\vartheta$, $K_3=1$  and $F_2(v)=\vartheta v e^{-\nu v}$. Then $F(u,v)\leq -K_1u+F_2(v)$ for $(u,v)\in \bar{\R}_+\times\bar{\R}_+$,  $F_2(v)\leq K_2 v$ for $v\in\bar{\R}_+$, and $H(u)\leq K_3 u$ for $u\in\bar{\R}_+$. So {\bf (A5)} is satisfied. Also $F_2(v)\le K_4=\vartheta/(\nu e)$ for $v\in\bar{\R}_+$ hence {\bf (A6a)} is satisfied. Then by Theorem \ref{thm5}, (\ref{71}) has at least one  positive steady state solution $(u_d(x),v_d(x))$ for any $d\in(0,d^*)$, and from Lemma \ref{lem4} part 1, any positive steady state satisfies $u_d(x),v_d(x)\leq \vartheta/(\nu \chi e)$.

     To prove \eqref{71} has no positive steady state solution for $d>d^*$, we notice that \eqref{72} implies that
     \begin{equation}\label{72a}
  \begin{cases}
    d\Delta u(x)-\chi u(x)+\vartheta v(x)\ge 0,&~x\in\Omega,\\
    \ds d\Delta v(x)+\frac{1}{\tau}u(x)-\frac{1}{\tau}v(x)=0,&~x\in\Omega,\\
    u(x)=v(x)=0,&~x\in\partial\Omega,
  \end{cases}
\end{equation}
 and on the other hand, $(w,z)=(\phi_1,M\vartheta\tau \phi_1)$ satisfies
 \begin{equation}\label{72b}
  \begin{cases}
    \ds d^*\Delta w(x)-\chi w(x)+ \frac{1}{\tau} z(x)= 0,&~x\in\Omega,\\
    \ds d^*\Delta z(x)+\vartheta w(x)-\frac{1}{\tau} z(x)=0,&~x\in\Omega,\\
    w(x)=z(x)=0,&~x\in\partial\Omega,
  \end{cases}
\end{equation}
where $d^*$ is defined in \eqref{dstar}. Multiplying the two equations in \eqref{72a} by $w$ and $z$, integrating and adding together, and subtracting the result of multiplying the two equations in \eqref{72b} by $u$ and $v$ and integrating and adding together, we obtain
\begin{equation*}
  0< (d-d^*)\int_{\Om}(\Delta w\cdot u+\Delta z\cdot v) dx=-(d-d^*)\la_1\int_{\Om} (wu+zv) dx,
\end{equation*}
which implies that $d<d^*$ as $u,v,w,z>0$.

2. Assume that a positive steady state solution  $(u_d(x),v_d(x))$ of \eqref{71} satisfies $v_d(x)\leq 1/\nu$. The linearized eigenvalue problem of \eqref{71} at $(u_d,v_d)$ is
 \begin{equation}\label{72c}
  \begin{cases}
    d\Delta \xi_1(x)-\chi \xi_1(x)+\vartheta e^{-\nu v_d(x)}(1-\nu v_d(x)) \xi_2 (x)=-\mu \xi_1(x),&~x\in\Omega,\\
    \ds d\Delta \xi_2(x)+\frac{1}{\tau}\xi_1(x)-\frac{1}{\tau}\xi_2(x)=-\mu \xi_2(x),&~x\in\Omega,\\
    \xi_1(x)=\xi_2(x)=0,&~x\in\partial\Omega.
  \end{cases}
\end{equation}
Since $v_d(x)\leq 1/\nu$, the system \eqref{72c} is cooperative in the sense that $F_v(u_d(x),v_d(x))=\vartheta e^{-\nu v_d(x)}(1-\nu v_d(x))>0$ and $G_u(u_d(x),v_d(x))=1/\tau>0$ (here $G(u,v)=(1/\tau)(u-v)$). Also the system \eqref{72c} is sublinear as
\begin{equation*}
\begin{split}
  &F(u_d,v_d)-u_dF_u(u_d,v_d)-v_dF_v(u_d,v_d)=\vartheta\nu v_d^2 e^{-\nu v_d}>0,\\
  &G(u_d,v_d)-u_dG_u(u_d,v_d)-v_dG_v(u_d,v_d)=0.
  \end{split}
\end{equation*}
Then from Theorem 2.3 of \cite{Cui2013}, the positive steady state solution  $(u_d(x),v_d(x))$ is locally asymptotically stable.
\end{proof}

In Proposition \ref{pro2}, the stability of positive steady state holds when the condition $v_d(x)\leq 1/\nu$ is satisfied. This is true when $d$ is close to $d^*$ (the bifurcation point), but it is not expected to be true when $d$ approaches to $0$. And the condition $v_d(x)\leq 1/\nu$ is also referred as the ``monotone" case for the Nicholson's blowfly model, while the ``non-monotone" case is the more complicated one.

As a numerical example, we consider (\ref{70}) with $\chi=0.8,~\vartheta=1,~\nu=0.6,~\tau=0.5,~\Omega=(0,\pi)$ and choose the initial condition $\eta(x,t)=0.1\sin{x},~t\in(-\infty,0)$. When $d=0.2>d^*=0.1362$, the zero solution $u\equiv 0$ is asymptotically stable, illustrated in Fig. \ref{Figure8} (A). However, when the  $d=0.1<d^*=0.1362$,  the zero solution loses its stability and a  positive steady state solution appears to be asymptotically stable as shown in Fig. \ref{Figure8} (B).

\begin{figure}[h]
\includegraphics[width=0.4\textwidth]{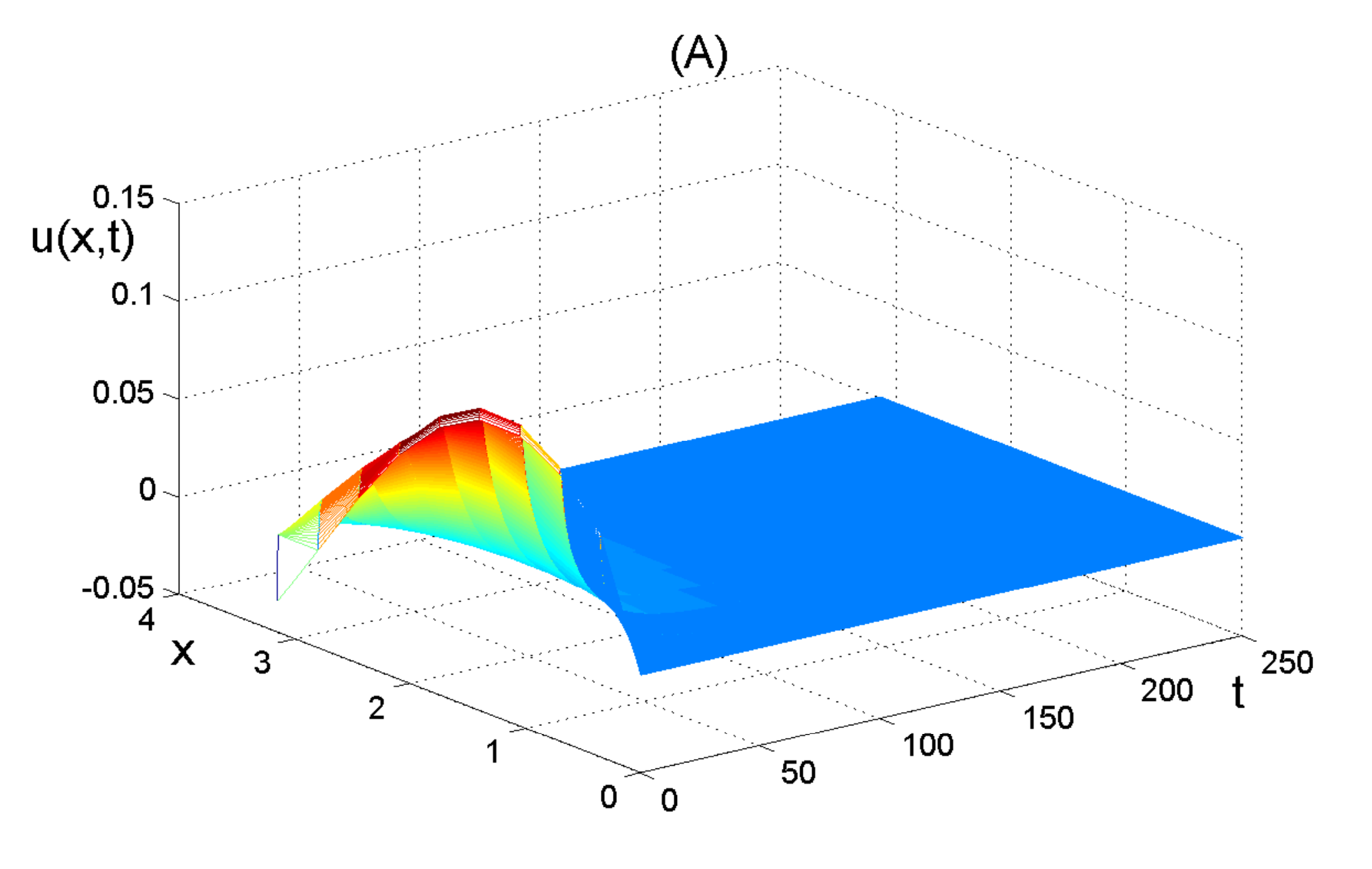}\hspace{0.1in}\includegraphics[width=0.4\textwidth]{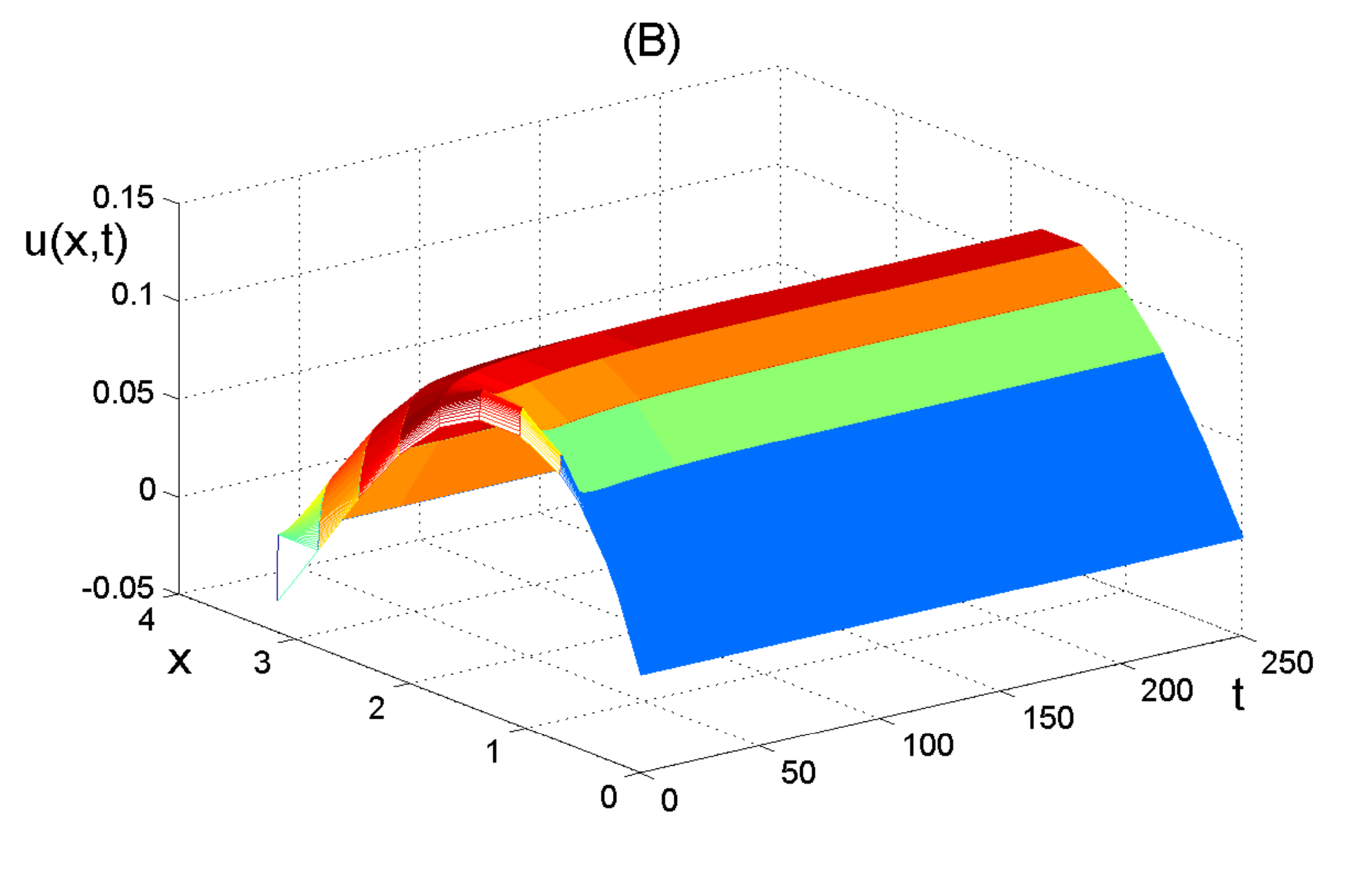}
\caption{Numerical simulations of (\ref{70}) with $\chi=0.8,~\vartheta=1,~\nu=0.6,~\tau=0.5,~\Omega=(0,\pi)$ and $\eta(x,t)=0.1\sin{x},~t\in(-\infty,0)$. (A): $d=0.2$, the solution converges to the trivial solution $u\equiv 0$; (B): $d=0.1$, the solution converges to a positive steady state solution.}\label{Figure8}
\end{figure}

A variant of the model \eqref{70} is
\begin{equation}\label{70a}
  \begin{cases}
    u_t(x,t)=d\Delta u(x,t)-\chi u(x,t)+\vartheta (g\ast\ast (ue^{-\nu u})(x,t)),&~x\in\Omega,~t>0,\\
    u(x,t)=0,&~x\in\partial\Omega,~t>0,
  \end{cases}
\end{equation}
and it is equivalent to
\begin{equation}\label{71a}
  \begin{cases}
    u_t(x,t)=d\Delta u(x,t)-\chi u(x,t)+\vartheta v(x,t),&~x\in\Omega,~t>0,\\
    \ds v_t(x,t)=d\Delta v(x,t)+\frac{1}{\tau}(u(x,t)e^{-\nu u(x,t)}-v(x,t)),&~x\in\Omega,~t>0,\\
    u(x,t)=v(x,t)=0,&~x\in\partial\Omega,~t>0.
  \end{cases}
\end{equation}
In this case, our theory in previous sections can also be applied with $F(u,v)=\chi u+\vartheta v$ and $H(u)=ue^{-\nu u}$, which satisfy {(\bf{A1})}-{(\bf{A3})}, {\bf (A1')}, {\bf (A5)} and {\bf (A6b)}. We can similarly prove
\begin{proposition}
  \label{pro3}
  Suppose that $\chi,\nu,\tau>0$ and $\vartheta>\chi$, and let $d^*$ be defined as in \eqref{dstar}. Then results in Proposition \ref{pro2} hold for \eqref{71a} except that the positive steady state $(u_d(x),v_d(x))$ satisfies $u_d(x)\leq \vartheta/(\chi\nu)$ and $v_d(x)\leq 1/\nu$ for all $0<d<d^*$, and if  $(u_d(x),v_d(x))$  satisfies $u_d(x)\leq 1/\nu$, then it is locally asymptotically stable.
\end{proposition}

Finally if we replace the Ricker type growth function $ue^{-\nu u}$ in \eqref{71} or \eqref{71a} by a Monod type (Holling type II) growth function $u/(A+u)$, much stronger results on the uniqueness and stability of positive steady state solution can be obtained. We use the model \eqref{71} as an example. Consider
\begin{equation}\label{79a}
  \begin{cases}
    \ds u_t(x,t)=d\Delta u(x,t)-\chi u(x,t)+ \frac{\vartheta (g\ast\ast u)(x,t)}{A+(g\ast\ast u)(x,t)},&~x\in\Omega,~t>0,\\
    u(x,t)=0,&~x\in\partial\Omega,~t>0,
  \end{cases}
\end{equation}
where $A>0$, and it is equivalent to
\begin{equation}\label{79b}
  \begin{cases}
    \ds u_t(x,t)=d\Delta u(x,t)-\chi u(x,t)+\frac{\vartheta v(x,t)}{A+ v(x,t)},&~x\in\Omega,~t>0,\\
    \ds v_t(x,t)=d\Delta v(x,t)+\frac{1}{\tau}(u(x,t)-v(x,t)),&~x\in\Omega,~t>0,\\
    u(x,t)=v(x,t)=0,&~x\in\partial\Omega,~t>0.
  \end{cases}
\end{equation}

\begin{proposition}
  \label{pro4}
  Suppose that $\chi,A,\tau>0$ and $\vartheta>\chi$, and let $d^*$ be defined as in \eqref{dstar}.
 Then system \eqref{79b} has a unique positive steady state solution $(u_d(x),v_d(x))$ for any $d\in (0,d^*)$ and has no positive steady state solution for $d>d^*$; the positive steady state $(u_d(x),v_d(x))$ satisfies $u_d(x),v_d(x)\leq \vartheta/\chi $ for all $0<d<d^*$; all positive steady state solutions of \eqref{79b} are on a curve $\Sigma_1=\{(d,u_d(\cdot),v_d(\cdot)): 0<d<d^*\}$ such that $\ds\lim_{d\to (d^*)^-}u_d(\cdot)=\lim_{d\to (d^*)^-}v_d(\cdot)=0$, and $(u_d(\cdot),v_d(\cdot))$ is globally asymptotically stable for $d\in(0,d^*)$.
\end{proposition}
\begin{proof}
We only prove the uniqueness and global stability of positive steady state solution as the other parts can be proved in a similar way as the proof of Proposition \ref{pro2}. Indeed in this case, the system \eqref{79b} is cooperative as $F_v(u,v)=\ds\frac{\vartheta A}{(A+v)^2}>0$ and $G_u(u,v)=1/\tau>0$, so the solutions of \eqref{79b} generate a semi-flow which is strongly monotone. The system \eqref{79b} is also sublinear (sub-homogeneous) as
\begin{equation*}
F(u,v)-uF_u(u,v)-vF_v(u,v)=\vartheta \ds\frac{v^2}{(A+v)^2}>0,\;\;
G(u,v)-uG_u(u,v)-vG_v(u,v)=0.
\end{equation*}
It is also easy to show the solutions of \eqref{79b} are ultimately uniformly bounded. Therefore from \cite[Theorem 2.3.2]{Zhao2017}, \eqref{79b} has a unique positive steady state that is globally attractive.
\end{proof}

\noindent{\bf Acknowledgements}
\small{This work was completed when the first author visited William \& Mary in 2015-2016, and she would like to thank W\&M for warm hospitality.}

\bibliographystyle{plain}
\bibliography{bifurcation-nonlocal,spatialtemporal}

\end{document}